\documentclass[11pt,reqno]{amsart}
\usepackage[margin=1in,letterpaper]{geometry}
\usepackage{graphicx}
\usepackage{amssymb}
\usepackage{amsthm}
\usepackage{mathrsfs}
\usepackage{stmaryrd}
\usepackage{accents} 
\usepackage{enumitem} 
\usepackage{subcaption}

\usepackage[bookmarksopen,bookmarksdepth=2]{hyperref} 
\allowdisplaybreaks

\makeatletter\let\over\@@over\makeatother

\numberwithin{equation}{section}
\theoremstyle{plain} 
\newtheorem{theorem}{Theorem}[section] 
\newtheorem{proposition}[theorem]{Proposition} 
\newtheorem{corollary}[theorem]{Corollary}
\newtheorem{lemma}[theorem]{Lemma}
\theoremstyle{remark}
\newtheorem{remark}[theorem]{Remark}
\theoremstyle{definition}

\newcommand{\be}{\begin{equation}}
\newcommand{\ee}{\end{equation}}%
\newcommand{\bse}{\begin{subequations}}
\newcommand{\ese}{\end{subequations}}

\newcommand{\sech}{\operatorname{sech}} 

\newcommand{\kernel}{\operatorname{ker}}
\newcommand{\linspan}{\operatorname{span}}

\newcommand{\id}{\operatorname{id}}

\newcommand{\LB}{\left[}
\newcommand{\RB}{\right]}
\newcommand{\LC}{\left(}
\newcommand{\RC}{\right)}

\newcommand{\R}{\mathbb{R}} 
\newcommand{\placeholder}{\,\cdot\,}
\newcommand{\maps}{\colon}         
\newcommand{\by}{\times}         
\newcommand{\sub}{\subset}         
\newcommand{\n}[2][]{#1\lVert #2 #1\rVert}
\newcommand{\abs}[2][]{#1\lvert #2 #1\rvert}

\newcommand{\bdd}{\mathrm{b}}       
\newcommand{\loc}{{\mathrm{loc}} }     
\newcommand{\odd}{\mathrm{o}}       

\newcommand\F{\mathscr F}
\newcommand{\genU}{\mathcal U}
\newcommand\strainW{\mathcal W}
\newcommand\Xspace{\mathscr X}
\newcommand\Yspace{\mathscr Y}
\newcommand{\cm}{{\mathscr C}}  
\newcommand\flowforce{\mathscr{H}}
\newcommand\limL{\mathscr{L}}

\newcommand\prineigenvalue{\sigma_{0}}
\newcommand{\unstableman}{\mathfrak{W}^{\mathrm{u}}}

\begin{document}

\title{Global bifurcation of anti-plane shear fronts}

\date{\today}

\author[R. M. Chen]{Robin Ming Chen}
\address{Department of Mathematics, University of Pittsburgh, Pittsburgh, PA 15260} 
\email{mingchen@pitt.edu}  

\author[S. Walsh]{Samuel Walsh}
\address{Department of Mathematics, University of Missouri, Columbia, MO 65211} 
\email{walshsa@missouri.edu} 

\author[M. H. Wheeler]{Miles H. Wheeler}
\address{Department of Mathematical Sciences, University of Bath, Bath BA2 7AY, United Kingdom}
\email{mw2319@bath.ac.uk}

\begin{abstract}
  We consider anti-plane shear deformations of an incompressible elastic solid whose reference configuration is an infinite cylinder with a cross section that is unbounded in one direction. For a class of generalized neo-Hookean strain energy densities and live body forces, we construct unbounded curves of front-type solutions using global bifurcation theory. Some of these curves contain solutions with deformations of arbitrarily large magnitude. 
\end{abstract}

\maketitle

\section{Introduction}

Consider an elastic solid whose undisturbed state is an infinite cylinder $\Omega \times\R$ where the coordinates are chosen so that the cross-section $\Omega := \R \times (-{\tfrac{\pi}{2}}, {\tfrac{\pi}{2}})$ lies in the $xy$-plane and the generator parallels the $z$-axis.  For simplicity, suppose that the lateral boundaries at $\{y = \pm\pi/2\}$ are held fixed.  \emph{Anti-plane shear} occurs when the solid is displaced out of the $xy$-plane and the deformation is independent of $z$. This leads to considerable analytical simplification since the full three-dimensional field equations can be reduced to a two-dimensional scalar elastostatic model.   Anti-plane shear is studied in connection to contact mechanics \cite{sofonea2009variational}, rectilinear steady flow of incompressible non-Newtonian fluids \cite{fosdick1973rectilinear}, structures with cracks \cite{paulino1993finite}, and phase transitions in solids \cite{silling1988consequences}, among many other areas.  From a purely mathematical perspective, this model is interesting as it is known to support a rich variety of nontrivial equilibria \cite{horgan1995anti}.  The existence of such solutions on bounded domains has been established by several authors \cite{sofonea2009variational,voss2018again} through variational methods. 

The present work concerns large anti-plane shear {\it fronts}, by which we mean  static equilibria where the displacement has distinct limits as $x \to -\infty$ and $x \to +\infty$.  While investigations of fronts are ubiquitous in the literature of reaction-diffusion equations and mathematical biology, for instance, they are largely unexamined in the context of elastostatics.  Recently, the existence of local curves of fronts lying in a neighborhood of the undisturbed state was proved via center manifold reduction techniques \cite{chen2019center}.  We now use analytic global bifurcation theory to extend these families into the non-perturbative regime.  Ultimately, this furnishes equilibria exhibiting deformation gradients of arbitrary magnitude, sometimes referred to as ``solutions in the large'' \cite{healey2019classical}.    

Global bifurcation has proven to be successful in treating a host of elasticity problems posed on bounded domains  \cite{healey1990symmetry,healey1997unbounded,healey1998global,healey2003global,healey2019classical}. In order to study fronts, though, one must naturally take $\Omega$ to be unbounded in the $x$-direction.   Classical bifurcation theory is ill-adapted to this setting, as it requires certain compactness properties that, at best, are very difficult to verify directly and, at worst, fail outright.  For example, standard hypotheses for degree theoretic global bifurcation are that the nonlinear operator is Fredholm index $0$ and locally proper  \cite{rabinowitz1971some}.  However, the linearized anti-plane shear equation at the undisturbed state in fact fails to be Fredholm.  Local properness of nonlinear elliptic operators on unbounded domains, moreover, is far from assured.  With that in mind, the authors developed a new, general global bifurcation theoretic approach  \cite{chen2018existence,chen2020global} specifically to analyze PDEs set on non-compact domains.  This machinery was then used to construct large-amplitude solitary water waves and hydrodynamics bores.  In part, our objective here is to demonstrate that these results have the potential to address questions of physical relevance in nonlinear elasticity.          
 
\subsection{The model}

\begin{figure}
  \centering
  \includegraphics[scale=1.1]{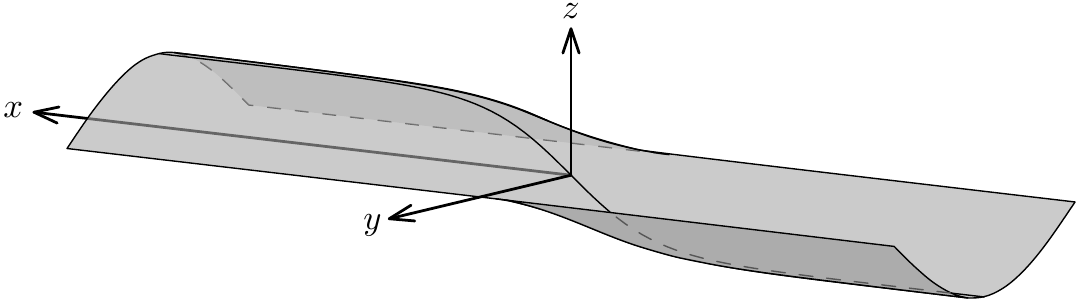}
  \caption{Sketch of a front solution satisfying the symmetry and monotonicity conditions in Theorem~\ref{main theorem}\ref{anti-plane blowup part}. The graph $z=u(x,y)$ is the image of the strip $\{z=0, \abs y < \tfrac{\pi}{2} \}$ under the anti-plane displacement $\id + u(x,y) e_z$.}
  \label{front figure}
\end{figure}
Let us now set down the governing equations.  For anti-plane shear, the displacement must have the form $\id + u(x,y) e_z$, where $e_z$ is the standard basis vector in the $z$-direction; see Figure~\ref{front figure}.  Further supposing the material is isotropic, hyperelastic, and incompressible, the principal invariants of the Cauchy--Green tensor are given by $I_1 = I_2 = 3 + |\nabla u|^2$ and $I_3 = 1$, and the strain energy density becomes $W = W(I_1, I_2)$. 

A characteristic feature of the the anti-plane shear problem is that it is in general overdetermined: the solution to the reduced scalar PDE, called the out-of-plane displacement, must also satisfy the other two in-plane equations. Within the context of both linear and nonlinear elasticity theory, necessary and sufficient admissibility conditions on the strain energy of various elastic materials have been derived which ensure that nontrivial anti-plane shear deformation can be sustained \cite{knowles1976finite,knowles1977note,jiang1991class,tsai1994anisotropic,horgan1994antiplane,pucci2013anti}.
Specifically, it is proved in \cite{knowles1976finite} that, under the so-called ``ellipticity condition'',
\begin{equation}\label{Knowles ellipticity}
{d \over dR} \LB R \left.\LC {\partial W \over \partial I_1}(I_1, I_2) + {\partial W \over \partial I_2}(I_1, I_2) \RC\right|_{I_1 = I_2 = 3 + R^2} \RB > 0 \quad \text{for all }\ R \ge 0,
\end{equation}
the energy function $W$ is admissible if and only if\footnote{It is further remarked in \cite{jiang1991class} that \eqref{Knowles criterion} remains sufficient for admissibility even in the absence of ellipticity \eqref{Knowles ellipticity}.} there exists some constant $k \in \R$ such that
\begin{equation}\label{Knowles criterion}
k {\partial W \over \partial I_1}(I_1, I_2) + (k-1) {\partial W \over \partial I_2}(I_1, I_2) = 0.
\end{equation}
It is also pointed out in \cite{voss2018again} that \eqref{Knowles ellipticity} is equivalent to the convexity of the mapping $\xi \in \R^2 \mapsto W(3+|\xi|^2, 3+|\xi|^2)$, which is assumed in most variational treatments of the problem. 

We restrict attention to generalized neo-Hookean materials \cite{horgan1995anti} with 
\begin{equation*}
  W = \overline{W}(I_1),
\end{equation*}
in which case \eqref{Knowles criterion} is automatically satisfied. Writing the strain energy density as a function of $|\nabla u|^2$ alone, 
\begin{equation}\label{def strain}
  \strainW(|\nabla u|^2) := \overline{W}(3 + |\nabla u|^2),
\end{equation}
the ellipticity condition \eqref{Knowles ellipticity} translates to
\begin{equation}\label{ellipticity on W}
  \strainW'(q) + 2\strainW''(q) q > 0 \quad \text{for }\ q \ge 0.
\end{equation}

Now we consider the system subject to an applied force which depends on the displacement $u$, that is, a ``live'' body force, or a live load \cite[Section 2.7]{ciarlet1988mathematical}. It is easy to see that in order to sustain nontrivial anti-plane shear, the body force can only be applied in the axial direction and needs to be independent of $z$. Following \cite{healey1998global}, we will consider the body force to be parameter dependent and denote the force density by $-b(u, \lambda) e_z$ for some parameter $\lambda$.

A static equilibrium then corresponds to a solution of the quasilinear PDE
\begin{equation}\label{anti-plane shear equation}
  \left\{ 
  \begin{aligned} 
    \nabla \cdot \left( \strainW'(|\nabla u|^2) \nabla u \right) - b( u,\lambda)  & =  0 & &  \text{in } \Omega,   \\
    u & = 0 \quad & & \text{on }  \partial\Omega.
  \end{aligned} 
  \right. 
\end{equation}
Here the homogeneous Dirichlet boundary condition simply means that the cylinder is clamped along its boundary. 
Condition \eqref{Knowles ellipticity} ensures that the PDE \eqref{anti-plane shear equation} is elliptic.  The case where \eqref{ellipticity on W} fails is quite interesting, particularly as it relates crack formation (see \cite[Section 6]{horgan1995anti} and the references therein), but is beyond the scope of the present paper.  

\subsection{Structural assumptions}

We are motivated by the example of a quadratic neo-Hookean material subjected to simple harmonic forcing:
\begin{equation} \label{motivational W and b} 
  \strainW(q) = q + w_1 q^2,
  \qquad 
  b(\varkappa,\lambda) = -(1+ \lambda) \varkappa,
\end{equation}
where here the constant $w_1 > 0$ and the parameter value $\lambda=0$ is ``critical'' in a sense which will be made precise later. Our results, however, apply to a much wider class of materials and forcings which satisfy the following structural conditions.

First we make the symmetry assumption that
\begin{equation} \label{b odd}
b(\placeholder, \lambda) \text{ is odd,} 
\end{equation}
and hence that \eqref{anti-plane shear equation} is invariant under the reflection $u \mapsto -u$. This greatly simplifies the analysis as we are then able to restrict attention to solutions $u$ which are odd in the unbounded variable $x$.

Next, we assume that both $\strainW$ and $b$ are analytic in their arguments. This allows us to use analytic global bifurcation theory, and also to make an expansion near the reference configuration at $\lambda=0$. We require this expansion to have the form
\begin{equation}\label{taylor} 
  \begin{aligned} 
    \strainW(q) & = q + w_1 q^2 + O(|q|^3),  \\  
    b(\varkappa, \lambda) &= -(1+\lambda)\varkappa  + 
    b_2 \varkappa^3  
    +
    O((|\varkappa| + |\lambda|^{1/2})^4),
  \end{aligned}
\end{equation}
where the constants $w_1$ and $b_2$ satisfy the strict inequality
\begin{align}
  \label{assumption on constants}
  b_2 + 2w_1 > 0.
\end{align}
This allows us to use the existence theory in \cite[Section 3]{chen2019center} for small solutions with $0 < \lambda \ll 1$. An expanded version of that result is given below in Section~\ref{small-amplitude section}.  

Finally, we require several global sign conditions. For the strain energy density, we impose the 
``enhanced'' ellipticity condition 
\begin{equation}\label{sufficient ellipticity on W}
  3\strainW''(q) + 2\strainW'''(q) q 
  =
  \big(\strainW'(q) + 2\strainW''(q) q \big)' 
  > 0 \quad \text{for }\ q > 0,
\end{equation}
which implies \eqref{ellipticity on W} since $\strainW'(0) = 1 > 0$ by \eqref{taylor}. For the body force, we suppose that
\begin{equation}\label{cond on b}
  \left\{\ \begin{aligned}
    & \textrm{for all } \lambda \geq 0, ~b(\placeholder, \lambda) \text{ is a strictly decreasing convex function on } (0,\infty) , \\
    & \textrm{for all } \varkappa > 0, ~b(\varkappa, \placeholder) \textrm{ is strictly decreasing and unbounded on } [0,\infty), \\ 
    & b_\varkappa(0, \lambda) < -1 \text{ for } \lambda > 0.
  \end{aligned}\right.
\end{equation}
The first two of these conditions encode the physically intuitive assumption that the magnitude of the body forcing increases as either the displacement or loading parameter is increased.   Together with the previous hypotheses, the inequalities \eqref{sufficient ellipticity on W} and \eqref{cond on b} guarantee that the set of $x$-independent solutions of \eqref{anti-plane shear equation} has a particularly simple structure, and \eqref{sufficient ellipticity on W} is further used to establish an important a priori bound. The conditions in \eqref{cond on b} only concern $\lambda \ge 0$ because, as we will see, the solutions we construct will all satisfy this inequality.  Finally, let us reiterate that the motivational choice of strain energy and body force \eqref{motivational W and b} satisfy all of the above requirements.

\subsection{Main results}

Our main result is the following.

\begin{theorem}[Global bifurcation of anti-plane shear fronts] \label{main theorem} 
  Suppose that the body force and strain energy satisfy the structural conditions \eqref{b odd}--\eqref{cond on b}. There exists a continuous curve $\cm$ of solutions to \eqref{anti-plane shear equation} admitting the $C^0$ parameterization
  \begin{equation*}
    \cm = \left\{ (u(s), \lambda(s)) : 0 < s < \infty  \right \} \subset C_\bdd^{3+\alpha}(\overline{\Omega}) \times (0,\infty)
  \end{equation*}
  with $(u(s),\lambda(s)) \to (0,0)$ as $s \to 0+$, and satisfying the following.
  \begin{enumerate}[label=\rm(\alph*)]
  \item \label{anti-plane monotone part} \textup{(Symmetry and monotonicity)} Each $(u(s), \lambda(s)) \in \cm$ is a strictly increasing monotone front with 
    \begin{equation}
      \begin{aligned} 
        \partial_x u(s) > 0 & \qquad \textup{in } \Omega,  \\
        \partial_y u(s) < 0 & \qquad \textup{for } x, y > 0.
      \end{aligned} \label{monotonicity} 
    \end{equation} 
    Moreover, $u(s)$ is odd in $x$ and even in $y$.
\item \label{anti-plane blowup part} \textup{(Unboundedness)} In the limit $s \to \infty$ we have blowup in that 
  \begin{equation}
   \| \partial_y u(s) \|_{C^0},~ \lambda(s) \longrightarrow \infty \qquad \textup{as } s \to \infty. \label{blowup alternative} 
  \end{equation}
\item \label{anti-plane analytic part} \textup{(Analyticity)} The curve $\cm$ is locally real-analytic.
  \end{enumerate}
\end{theorem}

Figure~\ref{front figure} shows a sketch of a front satisfying the symmetry and monotonicity conditions in \ref{anti-plane monotone part}.

\begin{remark}\label{reflection remark}
  Since \eqref{anti-plane shear equation} is invariant under reflections in $x$, one can reflect each of the solutions in $\cm$ to obtain a curve of strictly \emph{decreasing} monotone fronts with the inequalities in \eqref{monotonicity} reversed.
\end{remark}

\begin{remark}
One can relax many of these hypotheses at the cost of additional ambiguity regarding the limiting behavior along $\cm$.  For example, the assumption that $b(\varkappa, \placeholder)$ is unbounded on $[0,\infty)$ is used only in Lemma~\ref{uy blowup lemma}.  Without it, we would still have that $\lambda(s) \to \infty$, but not necessarily the blowup of $\partial_y u(s)$.  On the other hand, if \eqref{sufficient ellipticity on W} does not hold globally, then it may happen that the system loses ellipticity in the limit.  This scenario is of particular interest to crack formation but unfortunately is difficult to treat through our methodology.  In particular, it could coincide with any of the alternatives in Section~\ref{abstract bifurcation section}.   We also rely on \eqref{sufficient ellipticity on W} in our study of the conjugate flow problem in Section~\ref{conjugate section}.  If it does not hold, then there may exist multiple $x$-independent solutions of \eqref{anti-plane shear equation} that are conjugate in the sense of \eqref{def conjugate flows}.  Were this to occur, then we must allow for the possibility that a heteroclinic degeneracy develops in the limit along $\cm$; see Theorem~\ref{abstract global bifurcation theorem}\ref{gen hetero degeneracy}.  

Under the opposite sign conditions on the coefficients in \eqref{assumption on constants}, there exist spatially localized anti-plane shear equilibria for which $u(x,\placeholder)$ vanishes in the limits $x \to \pm\infty$.  Curves of small solutions of this form were also obtained in \cite{chen2019center}.  In a forthcoming paper, Hogancamp \cite{hogancamp2020broadening} continues these families globally by means of the general theory in \cite{chen2018existence}.  Unlike \eqref{blowup alternative}, he finds that the solutions broaden while remaining uniformly bounded in $C^{k+\alpha}$, for all $k \geq 0$.
\end{remark}

\begin{remark}
  Observe that while \eqref{blowup alternative} implies the existence of a monotone front solution for all $\lambda > 0$,  it does not give uniqueness.  Indeed, there may be many turning points along the global bifurcation curve.  There can also be secondary bifurcations resulting in branches of solutions not captured by Theorem~\ref{main theorem}.  On the other hand, $\cm$ is maximal among all locally analytic curves of increasing monotone front solutions containing the reference state $(u,\lambda) = (0,0)$.  Moreover, the curve comprises \emph{all} such solutions in a neighborhood of $(u,\lambda)=(0,0)$ in the sense that every sufficiently small front is an element of $\cm$ up to translation and reflection.
\end{remark}

\begin{remark}
  Here we are taking the regularity of $u$ to be slightly better than the classical one.  This choice simplifies the maximum principle arguments used to establish \eqref{monotonicity}. In fact, since $\strainW$ and $b$ are real analytic, a simple bootstrapping argument using elliptic theory shows that $u$ is $C_\bdd^{k+\alpha}$ for any $k \geq 0$ as soon as it is $C_\bdd^{2}$.  
\end{remark}

The rest of the paper is organized as follows. In Section \ref{preliminaries section}, we define the function spaces suited for the study of front-type solutions and the limiting linearized operators. In Section \ref{conjugate section}, we consider the so-called conjugate flow problem for the system.  This is crucial to the global continuation argument as it will allow us to characterize (and then rule out) the loss of compactness.  Next, in Section \ref{small-amplitude section}, we briefly recapitulate the small-amplitude existence result from \cite[Section 3]{chen2019center} and prove some additional facts.  Section~\ref{nodal section} is then devoted to establishing that certain monotonicity properties are preserved along closed sets of solutions extending these local curves.  Finally, these components are assembled in Section \ref{global section} to give the proof of Theorem \ref{main theorem}.  

\section{Preliminaries} \label{preliminaries section} 

\subsection{Function spaces}
We begin by setting down a functional analytic framework.  First, define 
\[ \Xspace  := \left\{ u \in C^{3+\alpha}(\overline{\Omega}) : u|_{\partial\Omega} = 0 \right\}, \qquad 
\Yspace  :=  C^{1+\alpha}(\overline{\Omega}).\]
Note that the elements of $\Xspace$ and $\Yspace$ are only locally H\"older continuous; we denote the corresponding space of uniformly H\"older continuous functions by $\Xspace_\bdd$ and $\Yspace_\bdd$, respectively.  We also will append a subscript of ``o'' to indicate the subspace of functions that are odd in $x$ and even in $y$. Finally we denote $\Xspace'$ be the subspace of $\Xspace$ consisting of functions independent of $x$, and likewise $\Yspace'$.  The PDE \eqref{anti-plane shear equation} can then be recast as the abstract operator equation
\[ \F(u,\lambda) = 0,\]
where $\F$ is a real-analytic mapping $\Xspace_\bdd \times \R \to \Yspace_\bdd$ and $\Xspace_{\bdd,\odd} \times \R  \to \Yspace_{\bdd,\odd}$.  

To study fronts, we follow the strategy of \cite{chen2020global} and conduct most of the work in the spaces 
\begin{equation}
  \begin{aligned}
    \Xspace_\infty & := \Big\{ u \in \Xspace_{\bdd,\odd} :  \lim_{x \to \pm\infty} \partial^\beta u \textrm{ exists for all } |\beta| \leq 3 \Big\}, \\
    \Yspace_\infty & := \Big\{ f \in \Yspace_{\bdd,\odd} : \lim_{x \to \pm\infty} \partial^\beta f \textrm{ exists  for all } |\beta| \leq 1 \Big\},
  \end{aligned}
  \label{infinity spaces definition} 
\end{equation}
with all the above limits are uniform in $y$.   Intuitively, $\Xspace_\infty$ is the largest closed subspace of $\Xspace_{\bdd}$ containing all front-type solutions with the desired symmetry properties (see, \cite[Lemma 2.3]{chen2020global}).  Since the solutions we construct will all have $\lambda > 0$, it is convenient to introduce the open set 
\begin{equation}
  \label{def U set}
  \mathcal{U}_\infty := \Xspace_\infty \times (0,\infty).  
\end{equation}
One can easily confirm that $\F \maps \mathcal{U}_\infty \to \Yspace_\infty$ is real analytic.

\subsection{Limiting operators and their principal eigenvalues}

If $(u,\lambda) \in \Xspace_\infty \times \R$ is a front-type solution of \eqref{anti-plane shear equation}, then the Fr\'echet derivative $\F_u(u,\lambda)$ is a linear elliptic operator whose coefficients have well-defined point-wise limits as $x \to \pm\infty$.  Taking those limits and restricting the domain to $x$-independent functions yields the so-called \emph{transversal linearized operators} at $x=-\infty$ and $x=+\infty$:
\begin{equation*}
  \limL_\pm^\prime(u,\lambda)\maps  \Xspace' \to \Yspace'
  \qquad v \mapsto \lim_{x \to \pm\infty} \F_u(u,\lambda) v. 
\end{equation*}
In our setting, these will be Sturm--Liouville-type ODE operators posed on the interval $\Omega' := (-\frac\pi 2, \frac\pi 2)$, and hence they possess principal eigenvalues that we denote by $\prineigenvalue^\pm(u,\lambda)$.  Recall that the principal eigenvalues bound the remainder of the spectrum from above and are characterized by the corresponding eigenfunction being strictly positive on $\Omega'$.  The problem here is self-adjoint, so the entirety of the spectrum lies on the real axis.   It is also worth noting that, if $u$ is odd in $x$, then the structure of the equation \eqref{anti-plane shear equation} implies that $\prineigenvalue^-(u,\lambda) = \prineigenvalue^+(u,\lambda)$.

\section{Conjugate flows}\label{conjugate section}

If  $(u,\lambda)$ is a front, then its limiting states $\lim_{x \to \pm\infty} u(x, \placeholder)$ are themselves $x$-independent solutions of the PDE \eqref{anti-plane shear equation}. This leads us to study the ODE 
\begin{equation}
 \label{U ODE}
  \left\{ \begin{aligned}
    \partial_y \big( \strainW'(U_y^2) U_y \big) - b(U, \lambda) & = 0 & \qquad & \textrm{in } \Omega^\prime = (-{\tfrac{\pi}{2}}, {\tfrac{\pi}{2}}),  \\
    U & = 0 & \qquad & \textrm{on } \partial\Omega^\prime = \left\{ y = \pm \tfrac{\pi}{2} \right\}.
  \end{aligned} \right.
\end{equation}
Given two solutions to the above ODE, a natural question is whether they can be connected by a front solution of the full anti-plane shear system. In this section we give a partial answer to this question,  which will be needed at several key points in the global bifurcation theoretic argument leading to Theorem~\ref{main theorem}.  

Our main tool is a conserved quantity of the full system.  Naturally, the anti-plane shear model \eqref{anti-plane shear equation} carries a variational structure in that it is formally given by 
\begin{equation}\label{var structure}
\delta \int_\Omega \mathcal L(u, \nabla u, \lambda) \,dx \, dy = 0
\end{equation}
with suitable boundary conditions on the admissible variations, and where the Lagrangian density
\begin{equation*}
  \mathcal L(u, \nabla u, \lambda) := {1\over 2} \strainW(|\nabla u|^2) + \mathcal{B}( u,\lambda).
\end{equation*}
Here,
\begin{equation*}
  \mathcal{B}(\varkappa,\lambda) := \int_0^\varkappa b(\tilde\varkappa,\lambda)\, d\tilde\varkappa
\end{equation*}
is the primitive of $b$. Note that \eqref{b odd} then implies that $\mathcal{B}(\placeholder, \lambda)$ is even.

Now, for $(u,\lambda) \in \Xspace_\bdd \times \R$, define the functional
\begin{equation*}
\begin{split}
  \flowforce(u, \lambda;x) &:= \int_{-\frac{\pi}{2}}^{\frac{\pi}{2}} \Big( \mathcal L(u,\nabla u,\lambda) - \mathcal L_\xi(u, \nabla u,\lambda)u_x \Big)\, dy
  \\&\phantom{:}= \int_{-\frac{\pi}{2}}^{\frac{\pi}{2}} \left( {1\over2}\strainW(|\nabla u|^2) - u_x^2 \strainW'(|\nabla u|^2) + \mathcal{B}(u, \lambda) \right)\, dy.
\end{split}
\end{equation*}
By \cite[Lemma 3.1]{chen2020global}, $\flowforce(u,\lambda;x)$ is independent of $x$ provided $(u,\lambda)$ is a solution to \eqref{anti-plane shear equation}.  This can of course also be verified by a direct calculation, and should be understood as a consequence of the variational structure \eqref{var structure} and the equation's translation invariance in $x$.  Following \cite[Definition 3.3]{chen2020global}, we call two distinct functions $U_\pm \in C^2(\Omega^\prime)$ {\it conjugate flows} if 
\begin{equation}
  \label{def conjugate flows}
  \F(U_-,\lambda) = \F(U_+,\lambda) = 0, \quad \textrm{and} \quad \flowforce(U_+, \lambda) = \flowforce(U_-, \lambda).
\end{equation}
That is, $(U_\pm,\lambda)$ both solve \eqref{U ODE} and are on the same level set of the conserved quantity $\flowforce$.  Naturally, being conjugate flows is a necessary condition for $U_+$ and $U_-$ to represent the limiting states of a front solution to \eqref{anti-plane shear equation}.

\begin{proposition}[Conjugate flows] \label{conjugate flow proposition} Assume that \eqref{b odd}, \eqref{sufficient ellipticity on W}, and \eqref{cond on b} hold. 
  \begin{enumerate}[label=\rm(\alph*)]
  \item \label{conjugate positive part} If $\lambda > 0$, there are exactly three solutions of the ODE \eqref{U ODE} which do not change sign: the unique positive solution $U=:U_+(\lambda)$, its reflection $U_-(\lambda) := -U_+(\lambda)$, and the trivial solution $U\equiv 0$.
  \item \label{conjugate uniqueness part} $U_+(\lambda)$ and $U_-(\lambda)$ are conjugate in the sense of \eqref{def conjugate flows}, but neither of them is conjugate to $U\equiv 0$.
  \item \label{conjugate zero part} If $\lambda = 0$, then the unique solution of \eqref{U ODE} is the trivial solution $U \equiv 0$.
  \end{enumerate}
\end{proposition}
In \ref{conjugate positive part} there will also exist many other solutions which change sign, but these do not play a major role in our analysis. Before giving the proof of this result, let us state and prove a corollary for the original problem of finding monotone front solutions to the PDE \eqref{anti-plane shear equation}.
\begin{corollary}[Limiting states of monotone symmetric fronts]\label{limiting corollary}
  Suppose that $(u,\lambda) \in \Xspace_\infty \times [0,\infty)$ is a front solution of \eqref{anti-plane shear equation} which is strictly monotone in that $\partial_x u > 0$ in $\Omega$. Then necessarily $\lambda > 0$ and $\displaystyle \lim_{x \to \pm\infty} u(x, \placeholder) =  U_\pm(\lambda)$.
\end{corollary}
\begin{proof}
  Since $u$ is odd in $x$, $u(0, y) = 0$ for $\abs y \leq \frac{\pi}{2}$. The monotonicity of $u$ in $x$ therefore implies that its limiting states $U_\pm := \lim_{x \to \pm\infty} u(x,\placeholder)$ satisfy $U_- < 0 < U_+$ for $\abs y < \frac{\pi}{2}$. In particular, $U_\pm$ are distinct and hence conjugate by the above discussion. If $\lambda = 0$, then Proposition~\ref{conjugate flow proposition}\ref{conjugate zero part} forces $U_- = U_+ = 0$, which is a contradiction. If $\lambda > 0$, then Proposition~\ref{conjugate flow proposition}\ref{conjugate positive part}--\ref{conjugate uniqueness part} imply that $U_\pm = U_\pm(\lambda)$ as desired.
\end{proof}

\begin{proof}[Proof of Proposition~\ref{conjugate flow proposition}]
  First we rewrite the interior equation of \eqref{U ODE} as a planar system
  \begin{equation}
    \label{U planar ODE}
    \left\{ \begin{aligned}
      U_{y} & = V, \\
      V_{y} & = \frac{b(U, \lambda)}{\strainW'(V^2) + 2\strainW''(V^2) V^2}.
    \end{aligned} \right.
  \end{equation}
  Here the denominator is strictly positive thanks to \eqref{ellipticity on W}, which follows in turn from the stronger condition \eqref{sufficient ellipticity on W}. Fix $\lambda > 0$.  It is easy to confirm from \eqref{taylor} and \eqref{cond on b} that the origin is the only rest point for this system, and the quantity
  \begin{equation}\label{Conserved UV} 
    H(U,V,\lambda) := \strainW'(V^2) V^2 - \frac{1}{2} \strainW(V^2) - \mathcal{B}(U, \lambda)  
  \end{equation}
  is conserved. Note that $\mathcal{B} \le 0$ and from \eqref{ellipticity on W} we have
  \[
    2q\strainW'(q) - \strainW(q) = \int^q_0 \LC \strainW'(s) + 2\strainW''(s) s \RC\,ds > 0 \quad \text{for } q > 0.
  \]
  Hence $H(U, V,\lambda) \ge 0$ and thus all orbits are periodic and centered on $(U,V) = (0,0)$.  

  Consider the level set 
  \[ H(U,V,\lambda) = c,\]
  for a fixed $c > 0$.  It will intersect the $V$-axis at $(0,\pm V_0(c))$ for some unique $V_0(c)>0$.  
  To find a positive solution to the boundary value problem \eqref{U ODE}, we look for a value of $c$ such that the orbit through $(0, V_0(c))$ arrives at $(0,-V_0(c))$ in time $\pi = |\partial \Omega^\prime|$.

It is convenient to switch to polar coordinates for the dependent variables:  
  \[ (U,V) \longmapsto (R, \Theta), \qquad U = R \cos{\Theta}, \quad V = R \sin{\Theta}.\]
  This transforms the planar system \eqref{U planar ODE} to 
  \begin{equation}
    \label{R Theta planar ODE}
    \left\{ \begin{aligned}
      R_{y} & = {1\over2} R \sin{(2\Theta)} + \frac{b(R\cos\Theta, \lambda) \sin\Theta}{f(R, \Theta)}, \\
      \Theta_{y} & = -\sin^2{(\Theta)} + \frac{b(R\cos\Theta, \lambda) \cos\Theta}{Rf(R, \Theta)},
    \end{aligned} \right.
  \end{equation}
  where, thanks to \eqref{ellipticity on W},
  \[
    f(R, \Theta) := \strainW'(R^2\sin^2\Theta) + 2 \strainW''(R^2\sin^2\Theta) R^2 \sin^2\Theta > 0.
  \]
  From \eqref{R Theta planar ODE} and the symmetry of the equation, we compute that the time required for the orbit to travel from $(0, V_0(c))$ to $(0, -V_0(c))$ is
  \begin{equation}
    \label{def period map}
    \begin{aligned}
      {P}(c,\lambda) & := 2 \int_{\frac{\pi}{2}}^0 {d\theta \over \Theta_y(r(\theta,c,\lambda) \sin\theta)} \\
      & =  2 \int_0^{\frac{\pi}{2}} \frac{r(\theta, c, \lambda) f(r(\theta, c, \lambda), \theta) }{ r(\theta, c,\lambda) f(r(\theta, c, \lambda), \theta) \sin^2\theta - b(r(\theta, c, \lambda)\cos\theta, \lambda) \cos\theta } \, d\theta,
    \end{aligned}
  \end{equation}
  where $r = r(\theta,c,\lambda) > 0$ is defined to be the unique positive solution to 
  \begin{equation}
    H(r\cos{\theta}, r\sin{\theta},\lambda) = c.\label{H level set} 
  \end{equation}
  We will often suppress the arguments of $r$ in the interest of readability.

  Differentiating \eqref{H level set} in $c$, we see that 
  \begin{align*} 1 & = {r_c \over r} \left(  q(\strainW'(q) + 2q \strainW''(q))   - b(\varkappa, \lambda) \varkappa  \right) \Big|_{\substack{ q = r^2 \sin^2{\theta} \\ \varkappa = r \cos{\theta}}}  
  \end{align*}
  and hence $r_c > 0$ by \eqref{ellipticity on W} and \eqref{cond on b}.  Likewise, differentiating the period map ${P}$ gives 
  \[
    {P}_c(c,\lambda) = 2 \int_0^{\frac{\pi}{2}} {\partial_c [\Theta_{y}(r \sin{\theta})] \over \Theta_y^2} \,d\theta.
  \]
  A direct computation shows that
  \begin{align*}
    \partial_c[\Theta_y(r \sin{\theta})]  & = - {r_c \cos\theta \over r^2 f^2(r,\theta)} \LC f(r,\theta) (b - \varkappa b_\varkappa) + 2b  q(3\strainW'' + 2q \strainW''')  \RC \Big|_{\substack{q = r^2 \sin^2{\theta} \\ \varkappa = r \cos{\theta}}}  > 0
  \end{align*}
  for $\theta \in (0,{\pi \over 2})$ by  \eqref{sufficient ellipticity on W} and \eqref{cond on b}. Therefore, we conclude that 
  \begin{equation}\label{P monotone}
    P_c(c,\lambda) > 0 \quad \text{for }\ \lambda > 0,
  \end{equation}
  which confirms that $P$ is strictly increasing in $c$. 

  On the other hand, sending $c \searrow 0$ in \eqref{def period map} we have for each $\lambda > 0$, 
  \begin{align*}
    \lim_{c \searrow 0} P(c,\lambda) & = \lim_{c \searrow 0} 2 \int_0^{\frac{\pi}{2}} \frac{f(r, \theta) }{  f(r, \theta) \sin^2\theta - {b(r\cos\theta, \lambda) \cos^2\theta \over r\cos\theta} } \, d\theta \\
    & =  2 \int_0^{\frac{\pi}{2}} { d\theta \over \sin^2\theta - b_\varkappa(0, \lambda) \cos^2\theta}  = {\pi \over \sqrt{-b_\varkappa(0,\lambda)}} < \pi,
  \end{align*}
  where we used \eqref{cond on b} in the last inequality. A solution to \eqref{U ODE} will correspond to the energy levels $c$ with
  \begin{equation}\label{period P}
    P(c,\lambda) = {\pi \over k}, \quad k = 1, 2, \ldots.
  \end{equation}
  Simply by the implicit function theorem there is, for each $k \geq 1$,
  a curve $c = c_k(\lambda)$ in the $(c,\lambda)$-plane satisfying \eqref{period P}. Since $P_c > 0$, it follows that
  \[
    c_1(\lambda) < c_2(\lambda) < \cdots. 
  \]
  We are only interested in the first of these, which gives rise to the unique positive solution $U_+(\lambda)$ of \eqref{U ODE} corresponding to $P(c,\lambda) = \pi$.  Repeating the above calculation at $\lambda = 0$ we find that the period map is still monotone in $c$ but since $b_\varkappa(0,0) = -1$, 
  \[
    \lim_{c \searrow 0} P(c,0) = \pi.
  \]
  Thus, in this case the only solution to \eqref{U ODE} is the trivial one $U = 0$. Since $b_{\varkappa\lambda}(0,0) = 1$ by \eqref{taylor}, this result extends to $-1 \ll \lambda < 0$.

  It remains to show \ref{conjugate uniqueness part}. That $U_+$ and $U_-$ are conjugate is clear from the symmetry of the equations, and by the above arguments, the only other solution of \eqref{U ODE} which does not change sign is the trivial one $U = U_\pm(0) \equiv 0$.  Since $\flowforce(0,\lambda) = 0$, it therefore suffices to show that $\flowforce(U_\pm(\lambda), \lambda) \neq 0$ for $\lambda > 0$.  To that end, denoting $\dot U_\pm := \partial_\lambda U_\pm$, we compute that
  \begin{align*}
  \frac{d}{d\lambda} \flowforce(U_\pm(\lambda), \lambda) & =  \frac{d}{d\lambda} \int_{-\frac{\pi}{2}}^{\frac{\pi}{2}} \left( \frac{1}{2} \strainW(U_{\pm y}^2) + \mathcal{B}(U_\pm, \lambda) \right) \, dy \\
  & = \int_{-\frac{\pi}{2}}^{\frac{\pi}{2}} \left( \strainW^\prime(U_{\pm y}^2) U_{\pm y} \dot U_{\pm y} + b(U_\pm,\lambda) \dot U_\pm + \mathcal{B}_\lambda(U_\pm, \lambda) \right) \, dy \\
  & = \int_{-\frac{\pi}{2}}^{\frac{\pi}{2}} \mathcal{B}_\lambda(U_\pm, \lambda) \, dy,
  \end{align*}
  where the last line follows from integrating by parts and the equation \eqref{U ODE} satisfied by $U$, and we have suppressed some $\lambda$ dependencies for readability.  By the sign condition \eqref{cond on b}, it then follows that $\lambda \mapsto \flowforce(U_\pm(\lambda), \lambda)$ is strictly decreasing.  This completes the proof. 
\end{proof}

\begin{remark} \label{limiting remark}
  The above proof gives us additional information about $U_+(\lambda)$. In particular, it is even and strictly decreasing for $y>0$.
\end{remark}

\section{Small-amplitude theory} \label{small-amplitude section}

In this section, we recall the existence theory for anti-plane shear fronts in a neighborhood of the undisturbed state.  This was first obtained in \cite[Section 3]{chen2019center}, but we will establish some further properties in preparation for the global continuation. Without loss of generality we only consider increasing fronts, as the analogous local curve of decreasing fronts can be obtained by simply reflecting in $x$ (cf.~Remark~\ref{reflection remark}).

\begin{theorem}[Small-amplitude fronts] \label{small amplitude theorem} 
  Let conditions \eqref{b odd}--\eqref{cond on b} hold. There exists $\varepsilon_0 > 0$ and a local $C^0$ curve 
  \[ 
    \cm_\loc = \left\{ (u^\varepsilon,\varepsilon^2) :  0 < \varepsilon < \varepsilon_0 \right\} 
    \subset \Xspace_\infty \times (0,\infty)
  \]
  of solutions to \eqref{anti-plane shear equation} with the following properties.
  \begin{enumerate}[label=\rm(\alph*)]   
  \item \textup{(Asymptotics)} The solutions on $\cm_\loc$ have the leading-order form
    \begin{equation}
      \label{local asymptotics}
      \begin{aligned}
        u^\varepsilon(x,y) & =  a_1 \varepsilon \tanh{\LC {\varepsilon x \over \sqrt{2}} \RC} \cos{(y)} + O(\varepsilon^2) \qquad \textup{in } C_\bdd^{3+\alpha}(\overline{\Omega}),
      \end{aligned}
    \end{equation}
    with $a_1 = 2/\sqrt{3(b_2+2w_1)}$. In particular, $\cm_\loc$ bifurcates from the trivial solution $(u,\lambda)=(0,0)$.  
  \item \textup{(Uniqueness)} \label{uniqueness part} In a neighborhood of $(0,0)$, all front solutions of \eqref{anti-plane shear equation} are, up to reflection and translation in $x$, contained in $\cm_\loc$.
  \item \textup{(Kernel)}\label{kernel part} The kernel of the linearized problem at $(u^\varepsilon, \varepsilon^2)$ is generated by $\partial_x u^\varepsilon$.  
  \end{enumerate}
\end{theorem}
\begin{proof}
  The existence of $\cm_\loc$, as well as the asymptotic information in \eqref{local asymptotics}, is proved in \cite[Theorem 3.1]{chen2019center} using a center manifold reduction approach. This theory constructs a sufficiently smooth coordinate map\footnote{The function $\Psi$ in \cite{chen2019center} in fact takes values in a certain weighted H\"older space on $\Omega$, but for our present purposes it is sufficient to restrict its output to the truncated domain $\Omega_1$.}  
  \begin{align*}
    \Psi \maps \R^3 \to C^{3+\alpha}(\overline{\Omega_1}), \qquad \textrm{where } \Omega_1 := (-1,1) \by \Omega^\prime,
  \end{align*}
  with the following properties. First, any solution along $\cm_\loc$ can be recovered from its trace $v := u(\placeholder,0)$ via
  \begin{equation}\label{u from v}
    u(x,y) = v(x) \varphi_0(y) + \Psi(v(x),v'(x),\varepsilon)(0,y),
  \end{equation}
  where here
  \begin{align*}
    \varphi_0(y) := \cos y
  \end{align*}
  is an element of the kernel of $\limL'(0,0)$. Second, the trace $v$ of any sufficiently small solution solves the second order ODE
  \begin{equation}\label{ap ODE}
    v'' = f(v, v', \varepsilon),
    \qquad 
    \text{where} \quad
    f(A, B, \varepsilon) := {d^2 \over dx^2} \Big|_{x=0} \Psi(A,B,\varepsilon)(x, 0).
  \end{equation} 
  By its construction and because of the symmetries of the problem, $\Psi$ satisfies 
  \begin{align}
    \label{Psi symmetry}
    \begin{aligned}
      \Psi(0,0,\varepsilon) &= 0 \text{ for all } \varepsilon,
      &
      \Psi(-A,B,\varepsilon) &= -\Psi(A,B,\varepsilon),
      \\
      \Psi_A(0,0,0) &= \Psi_B(0,0,0) = 0,
      &
      \Psi(A,-B,\varepsilon)(-x,y) &= \Psi(A,B,\varepsilon)(x,y).
    \end{aligned}
  \end{align}
  Finally, $\Psi$ has the expansion
  \begin{equation}\label{Psi expansion}
    \begin{aligned}
      \Psi(A,B,\varepsilon)(x,y)
      &= -\frac 12 x^2 \cos(y) \varepsilon^2 A
      + \left[ \frac{3b_2 + 6w_1}8 x^2 \cos y + \frac{b_2-6w_1}{32}(\cos y - \cos(3y)) \right] A^3\\
      &\qquad + O\Big( (\abs A + \abs B^{1/2} + \abs\varepsilon)^4 \Big).
    \end{aligned}
  \end{equation}

  Inserting \eqref{Psi expansion} into \eqref{ap ODE} and switching to the rescaled variables
  \begin{equation}\label{rescaling}
    x =: X/\varepsilon,
    \qquad 
    v(x) =: \varepsilon V(X),
    \qquad
    v_x(x) =: \varepsilon^2 W(X),
  \end{equation}
  we obtain the first order system
  \begin{equation}\label{planar reduced ODE} \left\{ \begin{aligned} 
    V_X & = W, \\
    W_X & = - V + a_1^{-2} V^3 + R(V, W, \varepsilon),
  \end{aligned} \right.
  \end{equation}
  with the rescaled error term $R(V,W,\varepsilon) = O\left( |\varepsilon| (|V| + |W|) \right)$. When $\varepsilon = 0$, \eqref{planar reduced ODE} has the explicit heteroclinic solution 
  \begin{equation*}
    V^0(X) = a_1 \tanh \frac X {\sqrt{2}}, \qquad W^0(X) = {a_1 \over \sqrt{2}} \sech^2 \frac X{\sqrt{2}}.
  \end{equation*}
  shown in Figure~\ref{phase portrait figure}. Since \eqref{planar reduced ODE} enjoys the symmetries $(V(X),W(X))\mapsto (V(-X),-W(-X))$ and $(V(X),W(X))\mapsto (V(X),-W(X))$, one can show that this heteroclinic connection persists for small $\varepsilon > 0$. Denote this solution by $(V^\varepsilon,W^\varepsilon)$, and the corresponding limits as $x \to \pm\infty$ by $(V_\pm^\varepsilon,0)$. Undoing the scaling and using \eqref{u from v}, one obtains both the existence of $\cm_\loc$ and the asymptotics \eqref{local asymptotics}.

  By the proof of Proposition~\ref{conjugate flow proposition}, all small solutions $(U,\lambda)$ of the transverse ODE \eqref{U ODE} besides the trivial solution $U=0$ have $\lambda = \varepsilon^2 > 0$ and $U = U_\pm(\lambda)$. Moreover, $U_+$ and $U_-$ are conjugate to one another but not to $U=0$. Thus all small fronts which are not constant in $x$ must have $U_\pm(\varepsilon^2)$ as their asymptotic states. In particular, $V_\pm^\varepsilon$ above is nothing other than $U_\pm(\varepsilon^2)(0)$.

  \begin{figure}
    \centering
    \includegraphics[scale=1.1]{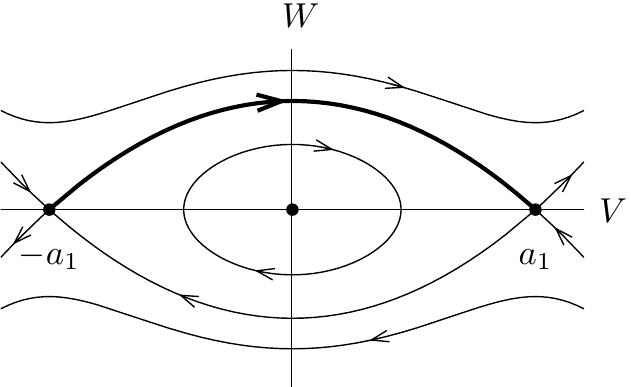}
    \caption{Phase portrait for the rescaled reduced ODE \eqref{planar reduced ODE} at $\varepsilon = 0$.  The thick curve is the heteroclinic orbit connecting the rest points $(-a_1,0)$ and $(a_1,0)$.}
    \label{phase portrait figure}
  \end{figure}

  We can now prove the uniqueness claimed in \ref{uniqueness part}. Suppose that $(u,\lambda)$ is a nontrivial front which is small in the sense that $\n u_{C^{3+\alpha}(\Omega)} + \abs \lambda \le \delta \ll 1$. By the above discussion we necessarily have $\lambda = \varepsilon^2 > 0$, and moreover $u \to U_\pm(\varepsilon^2)$ or $u \to U_\mp(\varepsilon^2)$ as $x\to\pm\infty$. Reflecting in $x$ if necessary, we can assume that we are in the first case. Choosing $\delta$ small enough, we can also guarantee that the trace $v = u(\placeholder,0)$ solves \eqref{ap ODE}. Rescaling, we obtain a solution $(V,W)$ of \eqref{planar reduced ODE}, with the same limiting states $(V_\pm^\varepsilon,0)$ as the orbit $(V^\varepsilon,W^\varepsilon)$ corresponding to $u^\varepsilon$. We claim that in fact $(V,W)$ is a translate of $(V^\varepsilon,W^\varepsilon)$. To prove this claim, consider the phase portrait of \eqref{planar reduced ODE}, which is qualitatively the same as the one shown in Figure~\ref{phase portrait figure}. The unstable manifold $\unstableman(V^\varepsilon_-,0)$ has two portions, one in the upper half plane and one in the lower half plane. If $(V,W)$ is not a translate of $(V^\varepsilon,W^\varepsilon)$, then its orbit must coincide with the portion of $\unstableman(V^\varepsilon_-,0)$ in the lower half plane. In particular, $V < V^\varepsilon_- < 0$ for $X \ll -1$. Since $V \to V^\varepsilon_+ > 0$ as $X \to +\infty$, there must be some $X_0$ where $V_X(X_0)=W(X_0)=0$. Shifting $X_0$ to the origin, reversibility implies that $(V(X),W(X)) = (V(-X),-W(-X))$. But then $(V,W)$ is a homoclinic orbit with $(V,W)\to(V^\varepsilon_-,0)$ as $\abs X \to \infty$, which is a contradiction. 
 
 Combining this uniqueness with the symmetry of \eqref{anti-plane shear equation} under reflections in $x$, $y$, and $u$, we immediately obtain that solutions along $\cm_\loc$ are odd in $x$ and even in $y$. Thus $\cm_\loc \sub \Xspace_\infty \by (0,\infty)$.
  
  To prove \ref{kernel part} we adopt the idea of \cite[Theorem 1.6]{chen2019center} which, in our setting, says that $\dot u$ is in the kernel of the linearized operator at $(u, \lambda)$ only if $\dot v := \dot u(\placeholder, 0)$ solves the linearized reduced equation
  \[
    \dot v'' 
    = \nabla_{(A,B)} f(v,  v', \lambda) \cdot ( \dot v,  \dot v').
  \]
  Equivalently, the corresponding rescaled quantities $(\dot V, \dot W)$ solve a nonautonomous planar system
  \[ 
    \begin{pmatrix} 
      \dot V_X \\ 
      \dot W_X 
    \end{pmatrix} 
    = 
    \begin{pmatrix}
      0 & 1 \\
      -1 + 3a_1^{-2} V^2 + R_V(V,W,\varepsilon) & R_W(V,W,\varepsilon)
    \end{pmatrix}
    \begin{pmatrix} 
      \dot V \\ 
      \dot W 
    \end{pmatrix}
    =: \mathcal{M}(X) 
    \begin{pmatrix} 
      \dot V \\ 
      \dot W 
    \end{pmatrix}.
  \] 
  Taking limits, we find that
  \begin{align*}
    \lim_{X\to\pm\infty}
    \mathcal M(X) &=
    \begin{pmatrix}
      0 & 1 \\
      -1 + 3a_1^{-2} V_\pm^2 + R_V(V_\pm,0,\varepsilon) & R_W(V_\pm,0,\varepsilon)
    \end{pmatrix}
    =
    \begin{pmatrix}
      0 & 1 \\
      2+ O(\varepsilon) & O(\varepsilon) 
    \end{pmatrix},
  \end{align*}
  and hence that $\mathcal{M}(X)$ is strictly hyperbolic for $|X| \gg 1$  with one negative and one positive eigenvalue.  A standard dynamical systems argument implies that there cannot be two linearly independent solutions of the reduced linearized problem that are uniformly bounded.  We may then conclude that the kernel of the linearized operator is indeed generated by $\partial_x u^\varepsilon$.
\end{proof}

Let us next consider the principal eigenvalues for the transversal linearized problems at infinity along $\cm_\loc$.   

\begin{lemma}[Local spectral non-degeneracy] \label{spectral nondegeneracy lemma} 
In the setting of Theorem~\ref{small amplitude theorem}, every solution $(u,\lambda) \in \cm_\loc$  with $0 < \lambda \ll 1$ is spectrally non-degenerate in that 
  \begin{equation} 
      \prineigenvalue^-(u,\lambda) =  \prineigenvalue^+(u,\lambda) < 0.\label{local spectral nondegen} 
  \end{equation}
\end{lemma}
\begin{proof}
  Note that $\prineigenvalue^{\pm}(0, 0) = 0$. Therefore, we must show that the principal eigenvalues perturb to the left  as $\lambda$ moves away from $0$ along $\cm_\loc$. As observed before, we need only consider $\sigma_0^+$ as the spectrum at $x=-\infty$ is the same. 

  For simplicity we will drop the $\pm$ and further write the limiting transversal linearized operator as 
  \[
    \limL^\prime(u,\lambda) \psi  = \partial_y \left( \left( \strainW'(U_y^2) + 2 U_y^2 \strainW''(U_y^2) \right) \partial_y \psi \right) - b_\varkappa(U, \lambda) \psi,
  \]
  where recall that $(U, \lambda)$ solves the problem
  \begin{equation}
    \label{trans prob}
    \F'(U,\lambda) := \partial_y\big( \strainW'(U_y^2) U_y \big) - b(U, \lambda) = 0.
  \end{equation}
  Clearly $\F'_U(U,\lambda) = \limL^\prime(u,\lambda)$.
  The proof of Theorem~\ref{small amplitude theorem} (or indeed the arguments in \cite{chen2019center}) shows that the solution $U$ of \eqref{trans prob} depends smoothly on $\varepsilon = \sqrt\lambda$. Moreover, sending $x \to \infty$ in \eqref{local asymptotics} yields the expansion
  \begin{align*}
    U(\varepsilon)(y) 
    = \lim_{x \to \infty} u^\varepsilon(x,y)
    = a_1 \varepsilon \cos y + O(\varepsilon^2) \qquad \textrm{in } C^{3+\alpha}([-\tfrac \pi 2,\tfrac \pi 2]).
  \end{align*}
  Using dots to denote derivatives in $\varepsilon$, we therefore have
  \begin{equation}
    \label{dot U}
    \dot U(0)(y) = a_1 \varphi_0 = a_1\cos y.
  \end{equation}
  Note that we are abusing notation somewhat by writing $U$ as a function of $\varepsilon$ rather than $\lambda$.

  With the asymptotics for $U$ in hand, we now turn to the eigenvalue problem, which in our notation is
  \begin{equation}
    \label{eigen prob}
    \LC \limL^\prime(U(\varepsilon),\varepsilon^2) - \prineigenvalue \RC \varphi = 0.
  \end{equation}
  We know that $(\varphi,\prineigenvalue,\varepsilon) = (\varphi_0,0,0)$ solves \eqref{trans prob}--\eqref{eigen prob}. By a familiar implicit function theorem argument, we deduce that there is a unique curve of nearby solutions, with $\varphi$ and $\prineigenvalue$ depending smoothly on $\varepsilon$.

  Differentiating \eqref{eigen prob} in $\varepsilon$, we find
  \begin{equation}
    \label{first deriv eigen prob}
    \limL'(U,\varepsilon^2) \dot \varphi + \dot \limL'(U,\varepsilon^2) \varphi = \dot\sigma_0 \varphi + \prineigenvalue \dot\varphi,
  \end{equation}
  where
  \[
    \dot\limL^\prime(U,\varepsilon^2) \psi := \partial_y \left( \LC 6\strainW'' + 4 U_y^2 \strainW''' \RC  U_y \dot U_y \partial_y \psi \right) -  \left(  b_{\varkappa\varkappa} \dot U + 2\varepsilon b_{\lambda\varkappa} \right) \psi
  \]
  and we are suppressing the arguments of $\strainW(U_y^2)$ and $b(U,\lambda)$ for readability. In particular, at $\varepsilon = 0$ where $U = 0$ and $\prineigenvalue = 0$, this becomes
  \[
    \dot \limL^\prime(0,0) \psi = - b_{\varkappa\varkappa}(0,0)\dot U(0) \psi = 0
  \]
  in light of \eqref{taylor}. Therefore multiplying \eqref{first deriv eigen prob} by $\varphi$, evaluating at $\varepsilon = 0$, and then integrating, we obtain
  \[
    \dot\sigma_0(0) = 0,
  \] 
  which forces us to proceed to higher order derivatives. 

  Differentiating \eqref{first deriv eigen prob} with respect to $\varepsilon$ we find that
  \begin{equation}
    \label{second deriv eigen prob}
    \limL'(U,\varepsilon^2) \ddot \varphi + 2 \dot \limL'(U,\varepsilon^2) \dot\varphi + \ddot \limL'(U,\varepsilon^2) \varphi = \ddot \sigma_0 \varphi + 2\dot\sigma_0 \dot\varphi + \prineigenvalue \ddot\varphi,
  \end{equation}
  where
  \begin{equation*}
    \begin{split}
      \ddot \limL'(U,\varepsilon^2) \psi &:= 
      \partial_y \Big[  
      \Big(  
      \big( 
      6\strainW'' + 4q \strainW''' \big)  \big( \dot U_y^2 + U_y \ddot U_y \big) +  \big( 20\strainW''' + 8q \strainW^{(4)} \big) U_y^2 \dot U_y^2  
      \Big) 
      \partial_y \psi \Big] \\
      &\qquad - \left( b_{\varkappa\varkappa\varkappa} \dot U^2 + b_{\varkappa\varkappa} \ddot U + 4\varepsilon b_{\lambda\varkappa\varkappa}\dot U + 4\varepsilon^2 b_{\lambda\lambda\varkappa} + 2b_{\lambda\varkappa} \right) \psi.
    \end{split}
  \end{equation*}
  At $\varepsilon = 0$, then, 
  \begin{equation*}
    \ddot \limL'(0,0) \psi 
    = 12 w_1 \partial_y \big(\dot U_y^2 \partial_y \psi \big) 
    - \big( 6b_2 \dot U^2 - 2  \big) \psi.
  \end{equation*}
  Multiplying \eqref{second deriv eigen prob} by $\varphi$ and integrating, we find that, at $\varepsilon=0$,
  \begin{equation*}
    \begin{split}
      \ddot \sigma_0(0) \|\varphi_0\|^2_{L^2} & = -12w_1\int^{\frac{\pi}{2}}_{-\frac{\pi}{2}} \dot U_y^2 (\partial_y\varphi_0)^2 \,dy - \int^{\frac{\pi}{2}}_{-\frac{\pi}{2}}  \LC 6b_2 \dot U^2 - 2  \RC \varphi_0^2 \,dy \\
      & = -12 a_1^2 w_1 \int^{\frac{\pi}{2}}_{-\frac{\pi}{2}}  (\partial_y \varphi_0)^2\,dy - 6 a_1^2 b_2 \int^{\frac{\pi}{2}}_{-\frac{\pi}{2}} \varphi_0^4 \,dy + 2 \|\varphi_0\|^2_{L^2},
    \end{split}
  \end{equation*}
  where here we have made use of \eqref{dot U} and the fact that $\varphi|_{\varepsilon=0} = \varphi_0 = \cos y$. Calculating the explicit integrals yields
  \[
    \ddot \sigma_0(0) = -2,
  \]
  which in turn proves that \eqref{local spectral nondegen} holds for $0 < \varepsilon \ll 1$.
\end{proof}

The final task of this section is to verify that the solutions on $\cm_\loc$ exhibit the monotonicity properties claimed in \eqref{monotonicity}.  For that purpose, we first state an elementary (but very useful) lemma.

\begin{lemma}[Nodal cone] \label{nodal cone lemma} Let $\mathscr{W}$ and $\mathscr{Z}$ be Banach spaces, and suppose that $\mathcal{N} \subset \mathscr{Z}$ is an open cone (that is, invariant under multiplication by strictly positive scalars).  If $G\maps \mathbb{R} \times \mathscr{W} \to \mathscr{Z}$ is continuous and
  \begin{equation}
    \frac{G(s,w)}{s} \to \psi \in \mathcal{N} \qquad \textup{as } (s,w) \to (0,0),~ s\neq0,\label{cone asymptotics} 
  \end{equation} 
  then, for all sufficiently small $(s,w)$ with $\pm s > 0$, we have $\pm G(s,w) \in \mathcal{N}$.
\end{lemma}
\begin{proof}
Since $\mathcal{N}$ is open, \eqref{cone asymptotics} implies that $G(s,w)/s \in \mathcal{N}$ for all $(s,w)$ sufficiently small and $s \neq 0$.  But then $|s| G(s,w)/s$ is also in $\mathcal{N}$, and so the result holds.
\end{proof}

Using this result, the desired monotonicity properties follow quickly from the reduction formula \eqref{u from v} and our understanding of the behavior of $v = u^\varepsilon(\placeholder, 0)$ obtained in the proof of Theorem~\ref{small amplitude theorem}.

\begin{lemma}[Local monotonicity] \label{local monotonicity lemma}
Each $(u, \lambda) \in \cm_\loc$ is a strictly increasing monotone front that exhibits the nodal properties:
  \begin{subequations}
  \label{local nodal properties}
\begin{alignat}{2}
u_x & > 0 \qquad \textup{in } \Omega, \label{local x monotonicity}  \\
u_y & < 0 \qquad \textup{in }   (0,\infty) \times (0, \tfrac{\pi}{2}]. \label{local y monotonicity}
\end{alignat}
  \end{subequations}
\end{lemma}
\begin{proof}  
Consider the $x$-directional monotonicity \eqref{local x monotonicity}.   In view of the center manifold reduction \eqref{u from v} and scalings in \eqref{rescaling}, we have that 
  \begin{equation}
    \label{u_x asymptotic form}
    \partial_x u^\varepsilon(x,y) = v^\prime(x) \varphi_0(y) + \Psi_A(v(x), v^\prime(x), \varepsilon)(0,y) v^\prime(x)  + \Psi_B(v(x), v^\prime(x), \varepsilon)(0,y) v^{\prime\prime}(x),
  \end{equation}
 where recall that $\varphi_0(y) = \cos{y}$, $v = \varepsilon V^\varepsilon(\varepsilon x)$, and $(V^\varepsilon,W^\varepsilon)$ solves the rescaled ODE \eqref{planar reduced ODE}.   From the phase portrait, we have seen that $v$ is strictly increasing (as $W^\varepsilon > 0$) and odd with $v(x) > 0$ for $x > 0$.  

Having \eqref{u_x asymptotic form} in mind, set 
\begin{align*}
\mathscr{W} & := \mathbb{R}^3, \\
\mathscr{Z} &  := \left\{ z \in C^1([-\tfrac{\pi}{2}, \tfrac{\pi}{2}]) : z(-\tfrac{\pi}{2}) = z(\tfrac{\pi}{2}) = 0 \right\}, \\
\mathcal{N} & := \left\{ z \in \mathscr{Z} : z > 0 \textrm{ on } (-\tfrac{\pi}{2}, \tfrac{\pi}{2}),~\mp z^\prime(\pm \tfrac{\pi}{2})  > 0 \right\},
\end{align*}
and consider the mapping $G\maps \mathbb{R} \times \mathscr{W} \to \mathscr{Z}$ defined by
\[ G(B; A, C, \varepsilon) := B \varphi_0 + \Psi_A(A,B,\varepsilon)(0,\placeholder) B + \Psi_B(A,B,\varepsilon)(0,\placeholder) C.\]
  Due to \eqref{Psi symmetry}, $B \mapsto \Psi_B(A,B,\varepsilon)(0,\placeholder)$ is odd.  Since $\Psi_B$ is smooth, there must therefore exist a continuous mapping $\Xi \maps \R^3 \to \mathscr Z$ satisfying
  \[ \Psi_B(A,B,\varepsilon)(0,\placeholder) = \Xi(A,B,\varepsilon) B.\]
Thus,
\[ \frac{G(B; A, C, \varepsilon)}{B} = \varphi_0 + \Psi_A(A,B,\varepsilon)(0,\placeholder) + \Xi(A, B, \varepsilon) C.\]
Sending $(B, A, C, \varepsilon) \to 0$, we find that the left-hand side limits to $\varphi_0 \in \mathcal{N}$.  Applying Lemma~\ref{nodal cone lemma}, we conclude that $\pm G(B; A, C, \varepsilon) \in \mathcal{N}$ for $\pm B > 0$ and $(B, A, C, \varepsilon)$ sufficiently small.  Comparing this to the asymptotics for $\partial_x u^\varepsilon$ in \eqref{u_x asymptotic form}, we must then have that \eqref{local x monotonicity} holds for $0 < \varepsilon \ll 1$.

The argument for the $y$-directional monotonicity is very similar. Again from the representation formula \eqref{u from v} one can compute that
\[
\partial_y u^\varepsilon(x,y) = v(x) \varphi_0^\prime(y) + \Psi_y(v(x), v^\prime(x), \varepsilon)(0,y).
\]
Let us redefine 
\begin{align*}
\mathscr{W} & := \mathbb{R}^2, \\
\mathscr{Z} &  := \left\{ z \in C^1([-\tfrac{\pi}{2}, \tfrac{\pi}{2}]) : z \textrm{ odd} \right\},  \\
\mathcal{N} & := \left\{ z \in \mathscr{Z} : z < 0 \textrm{ on } (0, \tfrac{\pi}{2}],~z^\prime(0)  < 0 \right\},
\end{align*}
and consider the map $G\maps \mathbb{R} \times \mathscr{W} \to \mathscr{Z}$ given by
\[ G(A; B, \varepsilon) := A \varphi_0^\prime + \Psi_y(A,B,\varepsilon)(0,\placeholder).\]
  By \eqref{Psi symmetry}, $A \mapsto \Psi_y(A,B, \varepsilon)(0,\placeholder)$ is odd, so we can find a smooth map $\Upsilon\maps \R^3 \to \mathscr Z$ such that 
  \[ \Psi_y(A,B,\varepsilon)(0,\placeholder) =  \Upsilon(A,B,\varepsilon) A. \]
Proceeding as before, one can then easily check that $G$ satisfies the hypothesis \eqref{cone asymptotics} of Lemma~\ref{nodal cone lemma}, and hence \eqref{local y monotonicity} holds for all $0 < \varepsilon \ll 1$.
\end{proof}

\section{Nodal pattern}\label{nodal section}

While Lemma~\ref{local monotonicity lemma} ensures that the small anti-plane shear fronts obtained in \cite{chen2019center} are strictly increasing, it is not obvious that this property should persist beyond the perturbative regime.  In \cite{chen2020global}, it is nevertheless shown that $\partial_x u$ has a fixed sign on any connected set of solutions that are spectrally non-degenerate in the sense of \eqref{local spectral nondegen}, provided it has a fixed sign at a single solution in this set.   In particular, this implies that the solutions obtained via global continuation will also be monotone in $x$.  The argument relies on repeated use of the maximum principle, and it applies to a broad class of elliptic PDE that includes the anti-plane shear model \eqref{anti-plane shear equation}.  Our equation enjoys an additional symmetry: invariance under reflection in $y$.  This will allow us to infer that $y$-directional monotonicity \eqref{local y monotonicity} is likewise preserved on the global bifurcation curve. Consider the fundamental half-strip
\[ R^+ := (0, \infty) \times (0,\tfrac{\pi}{2}),\]
whose boundary components we denote
\[ B^+ := [0,\infty) \times \{0\}, \qquad T^+ := [0,\infty) \times  \{ \tfrac{\pi}{2} \}, \qquad L^+ := \{0\} \times [0,\tfrac{\pi}{2}].\] 
The main result of this section is the following, where we recall that a front $(u,\lambda)$ is called \emph{strictly increasing} if $\partial_x u > 0$ in $\Omega$.

\begin{theorem}[Nodal properties] \label{nodal properties theorem} 
Suppose that $\mathcal{K} \subset \mathcal{U}_\infty$ is a connected set of strictly increasing monotone fronts and the spectral nondegeneracy condition \eqref{local spectral nondegen} holds along it.  If some $(u,\lambda) \in \mathcal{K}$ exhibits the nodal properties  
  \begin{subequations}
  \label{increasing nodal properties}
\begin{alignat}{2}
u_y & < 0 \qquad \textup{in } \overline{R^+} \setminus (L^+ \cup B^+),\label{u_y negative in R} \\
u_{yy} & < 0 \qquad \textup{on } B^+ \setminus \{(0,0)\}, \label{u_yy negative on interior B} \\
u_{xy} & < 0 \qquad \textup{on } L^+ \setminus \{ (0,0)\}, \label{u_xy negative on interior L} \\
u_{xyy}(0,0) & < 0,  \label{u_xyy corner sign}
\end{alignat}
  \end{subequations}
then every element of $\mathcal{K}$ satisfies \eqref{increasing nodal properties}.  
\end{theorem}

We start in the next lemma by establishing some basic information about the boundary behavior of $u$ and its derivatives on $R^+$.
\begin{lemma}[Boundary behavior] Let $(u,\lambda) \in \mathcal{U}_\infty$ be a strictly increasing monotone front solution of \eqref{anti-plane shear equation}.  Then  
  \begin{subequations} \label{preliminary nodal properties}
    \begin{alignat}{2}
      \label{vanishing on T}
      u, \, u_{yy} & = 0  \qquad \textup{on } T^+,\\
      \label{u_y decreasing on T}
      u_{xy} & < 0 \qquad \textup{on } T^+,\\
      \label{vanishing on L}
      u,\, u_{xx} & = 0 \qquad \textup{on } L^+,\\
      \label{vanishing on B}
      u_y & = 0 \qquad \textup{on } B^+.
    \end{alignat}
  \end{subequations}
\end{lemma}
\begin{proof}
  Most of these are immediate consequences of the boundary conditions or symmetry.  Indeed, $T^+ \subset \partial\Omega$, hence $u$ vanishes identically there.  This implies further that $\partial_x^k u = 0$ on $T^+$ for all $k \geq 0$, so evaluating the PDE \eqref{anti-plane shear equation} along the top yields
  \[ \LB \strainW'(u_y^2) + 2u_y^2\strainW''(u_y^2) \RB u_{yy} = 0 \qquad \textrm{on } T^+.\] 
  From \eqref{ellipticity on W}, this gives $u_{yy} = 0$ on $T^+$, proving \eqref{vanishing on T}. 
  The argument for \eqref{vanishing on L} is similar: $u$ (and hence $\partial_y^k u$) vanishes identically along $L^+$ by oddness.  Using this and the equation \eqref{anti-plane shear equation} gives
  \[ \LB \strainW'(u_x^2) + 2u_x^2\strainW''(u_x^2) \RB u_{xx} = 0 \qquad \textrm{on } L^+.\]

  That $u_y$ vanishes on $B$ is likewise a consequence the evenness of $u$ in $y$.  Finally, to obtain \eqref{u_y decreasing on T}, we observe that $u_x$ is in the kernel of the linearized operator at $(u,\lambda)$.  That is, $v := u_x$ satisfies 
  \begin{equation}
    \nabla \cdot \left( \strainW'(|\nabla u|^2) \nabla v + 2 \strainW''(|\nabla u|^2) (\nabla u \otimes \nabla u) \nabla v \right)   - b_\varkappa(u, \lambda) v = 0 \qquad \textrm{in } R^+. \label{linearized PDE} 
  \end{equation}
  By assumption, $u_x =v > 0$, and thus it attains its minimum along $T^+$.  Recalling \eqref{ellipticity on W} and \eqref{cond on b}, the sign of $u_{xy}$ in \eqref{u_y decreasing on T} then follows from the Hopf lemma (for positive solutions).  
\end{proof}

Using the above identities and the maximum principle, we can show that the full set of nodal properties \eqref{increasing nodal properties} can be collapsed to just transversal monotonicity \eqref{u_y negative in R}.

\begin{lemma}[Simplified nodal properties] \label{simplified nodal properties lemma} 
  Suppose that $(u,\lambda) \in \mathcal{U}_\infty$ is a strictly monotone increasing front solution of \eqref{anti-plane shear equation} that satisfies \eqref{u_y negative in R}.  Then $u$ also satisfies \eqref{u_yy negative on interior B}--\eqref{u_xyy corner sign}. 
\end{lemma}
\begin{proof}
  Differentiating the PDE \eqref{anti-plane shear equation} in $y$, we see that  $v = u_y$ satisfies the linear elliptic equation \eqref{linearized PDE}. Moreover, $u_y < 0$ in $R^+$ according to \eqref{u_y negative in R}, so $u_y$ attains its supremum on $R^+$ at $(0,0)$ thanks to  \eqref{vanishing on B}.  Combining \eqref{vanishing on L} and \eqref{vanishing on B}, we infer that
  \begin{equation}
    u_{yx}, \, u_{yy}, \, u_{yxx}, \, u_{yyy} = 0 \qquad \textrm{at } (0,0).\label{u_y origin expansion} 
  \end{equation}
  But this is in violation of the Serrin edge-point lemma as $\nabla u_y$ and $D^2 u_y$ cannot simultaneously vanish at a maximum; see, for example, \cite[Theorem E.9]{fraenkel2000introduction}.  Having arrived at a contradiction, we must have that \eqref{u_xyy corner sign} holds.  Likewise, since $u_y$ vanishes identically along $L^+ \cup B^+$ and is negative in $R^+$,  we obtain \eqref{u_yy negative on interior B} directly from the Hopf boundary-point lemma. The same argument shows that $u_{xy} < 0$ on $L^+$ except at the lower corner $(0,0)$ and possibly at the upper corner $(0,\frac{\pi}{2})$.  But notice that $u_y$ attains its minimum at $(0,\frac{\pi}{2})$, and moreover \eqref{vanishing on T} and \eqref{vanishing on L} give
  \begin{equation}
    u_{yy}, u_{yxx}, \, u_{yyx}, \, u_{yyy} = 0 \qquad \textrm{at } (0,\tfrac{\pi}{2}). \label{} 
  \end{equation}
  Thus the Serrin edge-point lemma ensures that $u_{yx}(0,\frac{\pi}{2}) < 0$, proving  \eqref{u_xy negative on interior L}.
\end{proof}

For later use, we observe that the above lemma together with Lemmas~\ref{spectral nondegeneracy lemma} and \ref{local monotonicity lemma} has the following immediate corollary. 
\begin{corollary}[Local nodal pattern] \label{local nodal corollary} 
Each $(u, \lambda) \in \cm_\loc$ is a strictly increasing monotone front that exhibits the full set of nodal properties \eqref{increasing nodal properties}.
\end{corollary}

The next two lemmas show that these nodal properties are (relatively) open and closed in a suitable topology. The main tools are \cite[Lemmas 2.7 and 2.8]{chen2020global}, which require the spectral nondegeneracy condition \eqref{local spectral nondegen}.

\begin{lemma}[Open property] 
  Suppose that $(\bar u,\bar \lambda) \in \mathcal{U}_\infty$ is a strictly increase  monotone front solution of \eqref{anti-plane shear equation} that satisfies \eqref{local spectral nondegen} and \eqref{increasing nodal properties}.  There exists $\delta = \delta(\bar u, \bar\lambda) > 0$ such that, if $(u,\lambda) \in \mathcal{U}_\infty$ is a solution with 
  \begin{equation}
    \| u - \bar u \|_{C^3(R^+)} + |\lambda - \bar\lambda | < \delta, \qquad \lim_{x \to \infty} u = U_+(\lambda), \label{u bar u close} 
  \end{equation}
  then $u$ is a strictly increasing monotone front and satisfies \eqref{increasing nodal properties}. 
\end{lemma}
\begin{proof}
  By Corollary~\ref{limiting corollary}, the limiting state as $x \to +\infty$ corresponding to  
  $(\bar u, \bar \lambda)$ is $\bar U_+ := U_+(\bar \lambda)$.  Now by Remark~\ref{limiting remark}, for any $\bar \lambda \geq \lambda_0 > 0$, we have 
  \[ \bar U_{+y} < 0 \quad \textrm{on } (0,\tfrac{\pi}{2}], \qquad  \bar U_{+y}(\tfrac{\pi}{2}) < -\delta_1, \qquad \bar U_{+ yy}(0) < -\delta_2,\]
  for some $\delta_1, \delta_2 > 0$ depending only on $\lambda_0$.  It follows from this, \eqref{u_y negative in R}, and \eqref{u_yy negative on interior B} that, for any $\varepsilon > 0$, there exists $\delta = \delta(\lambda_0, \varepsilon)$ so that 
  \[ u_y < 0 \qquad \textrm{on }  (\varepsilon, \infty) \times [0, \pi/2],\]
  for all $(u,\lambda)$ satisfying \eqref{u bar u close}.  Note that the principal eigenvalues depend continuously on the coefficients of the operator (see, for example, \cite[Lemma A.2]{chen2020global}), hence $\prineigenvalue^\pm(\bar u, \bar \lambda) < 0$ implies $\prineigenvalue^\pm(u, \lambda) < 0$ for $\delta$ small enough.  Perhaps shrinking it even further, we can then ensure that $(u,\lambda)$ is a strictly monotone front by applying \cite[Lemma 2.7]{chen2020global}.

  Consider next the sign of $u_y$ in a neighborhood of $L^+$.    As in the proof of Lemma~\ref{simplified nodal properties lemma},  we see that \eqref{vanishing on L} and \eqref{vanishing on B} imply that its derivatives at $(0,0)$ satisfy \eqref{u_y origin expansion}.  Expanding $u_y$ at the origin gives
  \[ u_y(x,y) = \frac{1}{2} u_{yyx}(0,0) xy + O(|x|^3 + |y|^3) \qquad \textrm{in } R^+.\]
  But $\bar u_{xyy}(0,0) < 0$ according to \eqref{u_xyy corner sign}, and so by perhaps further shrinking $\delta$, we can ensure that $u_y < 0$ on $\mathcal{O} \cap R^+$, for some open ball $\mathcal{O} \ni (0,0)$.   Let us now fix $\varepsilon$ to be half the radius of $\mathcal{O}$.  

  Similarly, choosing any point $(0,y_0) \in L^+ \setminus \mathcal{O}$, we have 
  \[ u_y(x,y) = u_{xy}(0,y_0) x + O(x^2 + (y-y_0)^2) \qquad \textrm{in } R^+.\]
  In view of \eqref{u_xy negative on interior L}, we may then shrink $\delta$ further so that 
  \[ u_y < 0 \qquad \textrm{on } \left(  (0, \varepsilon) \times [0,\pi/2] \right) \setminus \mathcal{O}.\]
  At last, then, we have shown that $u$ satisfies \eqref{u_y negative in R}.  Applying Lemma~\ref{simplified nodal properties lemma}, we conclude that $u$ exhibits all of the nodal properties \eqref{increasing nodal properties}.   
\end{proof}

\begin{lemma}[Closed property]  
  Suppose that $\{ (u_n, \lambda_n) \} \subset \mathcal{U}_\infty$ is a sequence of strictly monotone increasing front solutions to \eqref{anti-plane shear equation} that are each spectrally non-degenerate \eqref{local spectral nondegen} and $(u_n, \lambda_n) \to (u,\lambda)$ in $\Xspace_\bdd \times \R$ for some $(u,\lambda) \in \mathcal{U}_\infty$ also satisfying \eqref{local spectral nondegen}. If each $u_n$ exhibits the nodal properties \eqref{increasing nodal properties}, then so does $u$.
\end{lemma}
\begin{proof}
  Since \eqref{local spectral nondegen} holds by assumption, \cite[Lemma 2.8]{chen2020global} ensures that $(u,\lambda)$ is also a strictly increasing monotone front.  In view of Lemma~\ref{simplified nodal properties lemma}, it therefore suffices to show that $u$ satisfies \eqref{u_y negative in R}.   By continuity, we have that $u_y \leq 0$ in $\overline{R^+}$. We have already seen that $u_y$ satisfies the linear elliptic PDE \eqref{linearized PDE}, and since $\sup_{R^+}{u_y} = 0$, we may apply the maximum principle to conclude that $u_y < 0$ in $R^+$.  

  Fix a point $(x_0, \frac{\pi}{2})$ with $x_0 > 0$.  By monotonicity, $u$ is a positive solution of the elliptic PDE \eqref{anti-plane shear equation}, which can be viewed as linear.  Since $u$ attains its minimum value at $(x_0, \frac{\pi}{2})$, we have by the Hopf boundary-point lemma that $u_y < 0$ there.  Together with the previous paragraph, this proves that $u$ satisfies \eqref{u_y negative in R} and hence \eqref{increasing nodal properties}.  
\end{proof}

Combining these lemmas, the proof of Theorem~\ref{nodal properties theorem} is now immediate.

\section{Global continuation}\label{global section}

\subsection{Abstract global bifurcation theory} \label{abstract bifurcation section}

For the convenience of the reader, we record here the main tool for the proof of Theorem~\ref{main theorem}, which is the following general theorem on global bifurcation of monotone front solutions to elliptic PDE.   To simplify the presentation, we will only discuss its application to the specific problem \eqref{anti-plane shear equation}.  

Suppose that we have a local curve $\mathscr{K}_\loc$ of strictly monotone front solutions to \eqref{anti-plane shear equation}.  Due to translation invariance, for any $(u,\lambda) \in \mathscr{K}_\loc$, the linearized operator $\F_u(u,\lambda)$ will have at least one kernel direction.  As the first assumption, we require that the kernel is \emph{exactly} one dimensional:
\begin{equation}
  \ker \F_u(u,\lambda) = \linspan \{ \partial_x u \}.  \label{kernel assumption}\tag{H1} 
\end{equation}
We also impose a requirement on the spectrum of the transversal linearized operator:
\begin{equation}
  \prineigenvalue^\pm(u,\lambda)  < 0.  \label{spectral assumption}\tag{H2} 
\end{equation}
Note that, as observed before, the spectrum of the transversal limiting problems at $x = -\infty$ and $x =+\infty$ are identical, so there is no ambiguity above.  Lastly, assume that the local curve bifurcates from a singular point where the above spectral condition is violated.   That is, suppose $\mathscr{K}_\loc$ admits the $C^0$ parameterization
\[ 
\mathscr{K}_\loc = \left\{ \left( u(\varepsilon), \lambda(\varepsilon) \right) : 0 < \varepsilon < \varepsilon_0 \right\}  \subset \mathcal{U},
\]
where 
\begin{equation}
    (u(\varepsilon), \lambda(\varepsilon)) \to (u_0,\lambda_0) \in \mathcal{U} \quad \text{as } \varepsilon \to {0+}, \quad \textrm{and} \quad  
 \prineigenvalue^\pm(u_0,\lambda_0)=0.
  \label{local singular assumption} 
  \tag{H4} 
\end{equation}

Under the above hypotheses, we have the following global continuation result.  It corresponds to \cite[Theorem 1.2]{chen2020global} combined with \cite[Lemma 3.4]{chen2020global}.

\begin{theorem} \label{abstract global bifurcation theorem}
  Let $\mathscr{K}_\loc$ be a curve of strictly monotone front solutions to \eqref{anti-plane shear equation} bifurcating from a singular point as in \eqref{local singular assumption}.  Assume that at each $(u,\lambda) \in \mathscr{K}_\loc$, the nondegeneracy \eqref{kernel assumption} and spectral \eqref{spectral assumption} conditions hold.  

Then, possibly after translation, $\mathscr{K}_\loc$ is contained in a global curve of strictly monotone front solutions $\mathscr{K} \subset \mathcal{U}_\infty$, parameterized as 
\[ \mathscr{K} := \left\{ \left(u(s), \lambda(s) \right) : 0 < s < \infty  \right\} \subset \mathscr{F}^{-1}(0)\]
for some continuous $\R_+ \ni s \longmapsto (u(s), \lambda(s)) \in \mathcal{U}_\infty$ with the  properties enumerated below.
\begin{enumerate}[label=\rm(\alph*)]
\item \label{gen ift alternatives} \textup{(Alternatives)} As $s \to \infty$, one of three alternatives must occur:       
  \begin{enumerate}[label=\rm(A\arabic*)]
  \item \textup{(Blowup)} \label{gen blowup alternative}
    The quantity 
    \begin{align}
      \label{gen global blowup}
      N(s):= \n{u(s)}_{C^{3+\alpha}} + |\lambda(s)| + \frac 1{|\lambda(s)|} \longrightarrow \infty.
    \end{align}
  \item \label{gen hetero degeneracy} \textup{(Heteroclinic degeneracy)} There exist sequences $s_n \to \infty$ and $x_n \to \pm \infty$ with
    \[ \left(u(s_n)(\placeholder + x_n,\placeholder),\, \lambda(s_n) \right) \longrightarrow (u_*,\lambda_*) \textup{ in } C^{3}_\loc(\overline\Omega) \times \R^2\] 
    for some monotone front solution $(u_*,\lambda_*) \in \genU_\infty$, but the three limiting states
   \begin{equation*}
     \lim_{x \to \mp\infty} u_*(x, \placeholder),
     \quad 
     \lim_{n \to \infty} \lim_{x \to +\infty} u(s_n)(x, \placeholder),
     \quad 
     \lim_{n \to \infty} \lim_{x \to -\infty} u(s_n)(x, \placeholder),
   \end{equation*}    
   are all distinct and pairwise conjugate in the sense of \eqref{def conjugate flows}.
 \item \label{gen ripples} \textup{(Spectral degeneracy)}    There exists a sequence $s_n \to \infty$ with $\sup_{n} N(s_n) < \infty$ so that 
   \[
     \prineigenvalue^\pm(u(s_n), \lambda(s_n)) \to 0.  
   \]
  \end{enumerate}
\item \label{gen reconnect} For all $s$ sufficiently large, $(u(s),\lambda(s)) \not\in \mathscr{K}_\loc$.  In particular, $\mathscr{K}$ is not a closed loop.
\item \label{gen analytic} At each parameter value $s \in (0,\infty)$, $\mathscr{K}$ admits a local real-analytic reparameterization.
\end{enumerate}
\end{theorem}

\subsection{Spectral non-degeneracy} \label{spectral non-degeneracy section}

The next result will allow us to confirm that the spectral non-degeneracy persists not just for perturbative solutions, but globally.  

\begin{lemma}[Global non-degeneracy] \label{trivial kernel lemma} Suppose that $(u,\lambda) \in \mathcal{U}_\infty$ is a strictly increasing monotone front.  Then, $\kernel{\limL_\pm^\prime(u,\lambda)}$ is trivial.
\end{lemma}
\begin{proof}
  As the following argument does not involve the parameter $\lambda$, for simplicity of the presentation, we will suppress all dependence on it.  Also, as noted before, the transversal linearized problems at $x = \pm\infty$ are identical, so there is no need to append subscripts of $\pm$. In particular, we simply write $U$ for the limiting state $\lim_{x \to \pm\infty} u(x,\placeholder) = U_\pm(\lambda)$.

From \eqref{Conserved UV} we know that $U$ satisfies
\begin{equation}\label{Conserved U}
H(U, U_y) = \strainW'(U_y^2) U_y^2 - {1\over 2}\strainW(U_y^2) - \mathcal B(U) = \text{constant}.
\end{equation}
Moreover by \eqref{cond on b} and \eqref{sufficient ellipticity on W} we have
\begin{equation}\label{sign on Hderiv}
V H_V(U, V) > 0 \ \text{ for }\ V \ne 0.
\end{equation}

Let $\Phi = \Phi(y; \mu)$ be the unique solution to the initial value problem
\begin{equation}\label{ivp Phi}
\left\{ \begin{gathered}
\partial_y \big( \strainW'(\Phi_y^2) \Phi_y \big) - b(\Phi)  = 0  \quad  \textrm{in } \Omega^\prime, \\
\Phi(-\tfrac{\pi}{2}; \mu) = 0, \quad  \Phi_y(-\tfrac{\pi}{2}; \mu) = \mu.   
\end{gathered} \right.
\end{equation} 
Clearly $\Phi(y; U_{y}(-\tfrac{\pi}{2})) = U(y)$. Differentiating \eqref{ivp Phi} with respect to the parameter $\mu$ yields
\begin{equation}\label{ivp Phi_mu}
\left\{ \begin{gathered}
\partial_y \LB \LC \strainW'(U_{y}^2) + 2 U_{y}^2 \strainW''(U_{y}^2) \RC \dot\Phi_y \RB - b_\varkappa(U) \dot\Phi  = 0  \quad  \textrm{in } \Omega^\prime, \\
\dot\Phi(-\tfrac{\pi}{2}) = 0, \quad  \dot\Phi_y(-\tfrac{\pi}{2})  = 1,    
\end{gathered} \right.
\end{equation} 
  where $\dot \Phi(y) := \Phi_\mu(y; U_{y}(-\tfrac{\pi}{2}))$. If $\dot \Phi(\frac\pi 2) = 0$, then \eqref{ivp Phi_mu} is precisely the statement that $\dot \Phi \in \ker \limL'(u,\lambda)$. Since we are dealing with a boundary value problem for a linear second-order ODE, one can easily check that, conversely, $\ker\limL'(u, \lambda)$ is trivial whenever $\dot\Phi(\frac\pi 2) \ne 0$.

Recalling the definition of the period map $P$ in Section~\ref{conjugate section}, from \eqref{ivp Phi} it follows that
\begin{equation*}
\Phi(P(0, \mu); \mu) = 0.
\end{equation*}
Differentiating the above identity with respect to $\mu$ then gives
\begin{equation*}
\Phi_y \big(P(H(0,\mu); \mu) \big) P_c(H(0,\mu)) H_V(0, \mu) + \Phi_\mu \big( P(H(0,\mu)); \mu \big) = 0,
\end{equation*}
which at $y = \tfrac{\pi}{2}$ and $\mu = U_y(-\tfrac{\pi}{2})$ becomes
\begin{equation*}
U_y(\tfrac{\pi}{2}) P_c\big( H(0, U_y(-\tfrac{\pi}{2})) \big) H_V(0, U_y(-\tfrac{\pi}{2})) + \dot\Phi(\tfrac{\pi}{2}) = 0.
\end{equation*}
By \eqref{sign on Hderiv} and \eqref{P monotone}, the first term  above is negative.   Thus,  $\dot \Phi(\tfrac{\pi}{2}) < 0$, which proves that $\ker{\limL'}(u, \lambda)$ is trivial. 
\end{proof}

\subsection{Uniform regularity}

The purpose of this section is to derive a priori estimates on front solutions to the anti-plane shear equation that will eventually allow us to conclude that $\lambda$ is unbounded along the global curve. Following \cite[Chapter 7]{sperb1981maximum}, we first show that the divergence structure of \eqref{anti-plane shear equation} and the uniform ellipticity condition \eqref{ellipticity on W} imply that there is an auxiliary function $\mathcal{P}(u,|\nabla u|^2, \lambda)$ that obeys a maximum principle. By exploiting certain coercivity properties of $\mathcal{P}$, this will allow us to control $u$ in terms of $\lambda$.  

More precisely, define 
\begin{equation}\label{P function}
\mathcal{P}(\varkappa,q, \lambda) := 2q \strainW'(q) - \strainW(q) - 2 \mathcal{B}(\varkappa, \lambda). 
\end{equation}
An application of \cite[Theorem 7.1]{sperb1981maximum} gives
\begin{lemma}[Maximum principle for $\mathcal{P}$]\label{gradient bounds lemma}
Assume that \eqref{ellipticity on W} holds and let $(u,\lambda) \in \Xspace_\bdd \times \R$ be a solution to \eqref{anti-plane shear equation}. Then $\mathcal{P}(u,|\nabla u|^2,  \lambda)$ cannot achieve an interior maximum or minimum except at critical points of $u$. 
\end{lemma}

Now, from \eqref{sufficient ellipticity on W} 
we see that 
\[
2q \strainW''(q) + \strainW'(q) > \strainW'(0) = 1 \qquad \text{for }\ q > 0.
\]
Integrating this leads to 
\begin{equation*}
2q \strainW'(q) - \strainW(q) > q \qquad \text{for }\ q > 0.
\end{equation*}
From the definition of $\mathcal{P}$ \eqref{P function} and negativity of $\mathcal{B}$ \eqref{cond on b}, we may then infer the bound
\begin{equation}
\mathcal{P}(u,|\nabla u|^2, \lambda) \ge |\nabla u|^2.
\label{P coercivity}
\end{equation}

Supposing now that $(u,\lambda)$ is a strictly monotone front, Lemma~\ref{gradient bounds lemma} implies that $\mathcal{P}(u, |\nabla u|^2, \lambda)$ is maximized at infinity.  Combining this with the above inequality and a bootstrapping argument, we obtain the following. 

\begin{theorem}[Uniform regularity] \label{uniform regularity theorem} Every monotone front solution $(u,\lambda) \in \mathcal{U}_\infty$ to \eqref{anti-plane shear equation} with $|\lambda| < \Lambda$ obeys the a priori bound
\[ \| u \|_{C^{3+\alpha}(\Omega)} < C\]
where the constant $C = C(\Lambda) > 0$.
\end{theorem}
\begin{proof}
Throughout the course of the proof, let $C$ denote a generic positive constant depending only on $\Lambda$.  First, notice that because $u$ is a strictly monotone front, it is maximized and minimized only at infinity.  Thus, by Corollary~\ref{limiting corollary},
  \begin{equation}
    \| u \|_{C^0(\Omega)} = \sup_{\Omega^\prime}{|U_\pm(\lambda)|} \leq C.\label{C0 bound} 
  \end{equation}
On the other hand, thanks to Lemma~\ref{gradient bounds lemma} and \eqref{P coercivity}, we have
\[ \| \nabla u \|_{C^0(\Omega)} \leq \| \mathcal{P}(u, |\nabla u|^2, \lambda) \|_{C^0(\Omega)} = \sup_{\Omega^\prime} \mathcal{P}(U_\pm(\lambda), |\partial_y U_\pm(\lambda)|^2,\lambda) \leq C. \label{C1 bound} \]
Together with \eqref{C0 bound}, this gives control of $\| u \|_{C^1}$.

To upgrade this to the full $C^{3+\alpha}$ norm, we make use of a familiar elliptic regularity argument.  Observe that $\partial_x u$ can be thought of as $W^{2,\infty}(\Omega)$ strong solution of  \eqref{linearized PDE}.  By the previous paragraph, the coefficients of this PDE are bounded uniformly in $L^\infty(\Omega)$ in terms of $\Lambda$.  Using the De Giorgi--Nash-type estimate \cite[Corollary 9.29]{gilbarg2001elliptic}, we find that
\[ \| \partial_x u \|_{C^\alpha(\Omega_m)} \leq C \| \partial_x u \|_{C^0(\Omega)} < C,\]
for any $m \in \mathbb{Z}$, where $\Omega_m := (m, m+1) \times \Omega^\prime$.  Note that the constant above is independent of $m$, and so this implies further that $\| \partial_x u\|_{C^\alpha(\Omega)} < C$. 
The same reasoning gives an equivalent bound for $\partial_y u$, and hence $\| u \|_{C^{1+\alpha}} < C$.

Now, we can simply apply linear Schauder theory to \eqref{anti-plane shear equation} to see that 
\[ \| u \|_{C^{2+\alpha}(\Omega)} \leq C.\]
Bootstrapping the above argument once more leads to the desired $C^{3+\alpha}$ bound, completing the proof.
\end{proof}

\begin{lemma}[Deformation gradient blowup] \label{uy blowup lemma} It holds that 
\[ \| \partial_y U_\pm(\lambda) \|_{C^0} \longrightarrow \infty \qquad \textup{as } \lambda \to \infty.\]
\end{lemma}
\begin{proof}
First, we observe that $\lambda \mapsto U_+(\lambda)|_{y=0}$ is strictly increasing.  This simply follows from the fact that  $\dot U_+ := \partial_\lambda U_+(\lambda)$ solves 
\[ \limL^\prime(U_+(\lambda),\lambda) \dot U_+ =  -b_\lambda(U_+(\lambda),\lambda)\]
and vanishes on $\partial \Omega^\prime$.  In view of \eqref{cond on b}, the right-hand side above is strictly positive.  On the other hand, Lemma~\ref{spectral nondegeneracy lemma} and Lemma~\ref{trivial kernel lemma} ensure that the principal eigenvalue of $\limL^\prime(U_+(\lambda),\lambda)$ is strictly negative for all $\lambda > 0$.  We may therefore apply the maximum principle to conclude that $\dot U_+ > 0$.   

Next, we recall that the equation satisfied by $U_+$ can be rewritten as the planar system \eqref{U planar ODE} which has the first integral $H$ given by \eqref{Conserved UV}.  Since $U_+(\lambda)$ is odd, writing $V_+(\lambda) := U_{y+}(\lambda)$, we have that
\[ c(\lambda) = H(U_+(\lambda), V_{+}(\lambda)) = -\mathcal{B}(U_+(\lambda), \lambda)|_{y =0} > -\mathcal{B}(U_+(\bar\lambda), \lambda)|_{y = 0}\]
for $0 < \bar\lambda < \lambda$ by \eqref{cond on b}.  Fixing a $ \bar\lambda > 0$, the unboundedness of $b(\varkappa, \placeholder)$ implies that  $c(\lambda) \to \infty$ as $\lambda \to \infty$.  On the other hand, evaluating $H$ at the top of the domain reveals that
\[ c(\lambda) = \left( \strainW^\prime(V_+(\lambda)^2) V_+(\lambda)^2 - \frac{1}{2} \strainW(V_+(\lambda)^2)\right)\Big|_{y = \frac{\pi}{2}}. \]
Thus,  $\| \partial_y U_{+}(\lambda) \|_{C^0} \to \infty$ as $\lambda \to \infty$, which completes the proof.
\end{proof}

\subsection{Proof of the main result}

Finally, we turn to Theorem~\ref{main theorem}.  Recall that we have constructed a local curve $\cm_\loc$ of small-amplitude strictly increasing fronts in Section~\ref{small-amplitude section}.

\begin{proof}[Proof of Theorem~\ref{main theorem}]  We have already verified in Theorem~\ref{small amplitude theorem}\ref{kernel part} that the kernel condition \eqref{kernel assumption} holds along $\cm_\loc$.   Lemma~\ref{spectral nondegeneracy lemma}, moreover, implies that it satisfies the spectral assumption \eqref{spectral assumption}.   By construction, $\cm_\loc$ bifurcates from $(0,0)$, and hence \eqref{local singular assumption} holds.  Lastly, in view of Corollary~\ref{local nodal corollary}, the solutions on $\cm_\loc$ are strictly monotone increasing and exhibit the nodal properties \eqref{increasing nodal properties}.   

We are therefore justified in applying Theorem~\ref{abstract global bifurcation theorem} with $\mathscr{K}_\loc = \cm_\loc$, furnishing a global bifurcation curve $\cm$.  The symmetry properties claimed in part~\ref{anti-plane monotone part} are encoded in the definition of the spaces \eqref{infinity spaces definition}, while the monotonicity \eqref{monotonicity} follows from the fact that the fronts in  $\cm_\loc$ are strictly increasing and Theorem~\ref{nodal properties theorem}.  The local real analyticity of $\cm$ asserted in part~\ref{anti-plane analytic part} is likewise a direct consequence of Theorem~\ref{abstract global bifurcation theorem}\ref{gen analytic}.

  Consider now the limiting behavior along $\cm$.  Suppose for the sake of contradiction that \ref{gen hetero degeneracy} occurs, and without loss of generality assume that the sequence $x_n \to +\infty$. By the monotonicity properties in \eqref{monotonicity}, the limiting front $(u_*,\lambda_*)$ has $u_* \ge 0$ in $\Omega$. Thus its limiting states
  \begin{align*}
    \lim_{x \to \pm\infty} u_*(x,\placeholder) \ge 0.
  \end{align*}
  By Proposition~\ref{conjugate flow proposition}\ref{conjugate positive part}, these states must therefore be either $U_+(\lambda_*)$ or $0$. However, the three limiting states
  \begin{align*}
    \lim_{x \to -\infty} u_*(x,\placeholder)
    \quad\text{and}\quad
    \lim_{n \to \infty} \lim_{x \to \pm\infty}u(s_n)(x,\placeholder)
    = U_\pm(\lambda_*),
  \end{align*}
  are distinct and pairwise conjugate, and so this is impossible in view of Proposition~\ref{conjugate flow proposition}\ref{conjugate uniqueness part}.
  
  Reconnection to the trivial solution $(0,0)$ is ruled out by Theorem~\ref{small amplitude theorem}\ref{uniqueness part} and Theorem~\ref{abstract global bifurcation theorem}\ref{gen reconnect}. We claim further that the spectral degeneracy alternative \ref{gen ripples} can only happen in conjunction with blowup of $\lambda$ as in \eqref{blowup alternative}.   To see this, note that because the curve does not reconnect, if $\lambda$ is uniformly bounded along it, then spectral degeneracy would imply there exists $0 < \lambda_* < \infty$ for which the kernel of $\F_U^\prime(U_+(\lambda_*), \lambda_*)$ is nontrivial.  But this is impossible due to Lemma~\ref{trivial kernel lemma}. 

Thus we are left only with blowup as in \eqref{gen global blowup}.  On the other hand, $\liminf_{s \to \infty} \lambda(s) > 0$ by the previous paragraph, and so Theorem~\ref{uniform regularity theorem} ensures that $\lambda(s) \to \infty$.  Finally, Lemma~\ref{uy blowup lemma} implies that this also leads to blowup in the deformation gradient:
\[ \| \partial_y u(s) \|_{C^0(\Omega)} \geq \sup{|U_{y\pm}(\lambda(s))|} \longrightarrow \infty \qquad \textrm{as } s \to \infty,\]
completing the proof of part~\ref{anti-plane blowup part}.
\end{proof}

\section*{Acknowledgements} 

The research of RMC is supported in part by the NSF through DMS-1907584.  The research of SW is supported in part by the NSF through DMS-1812436.

\bibliographystyle{siam}
\bibliography{projectdescription}

\end{document}